\documentclass[11pt]{amsart} \textwidth=14.5cm \oddsidemargin=1cm
\usepackage[rgb]{xcolor}
\usepackage{tikz}
\usepackage{verbatim}
\usepackage{amsmath}
\usepackage{amsxtra}
\usepackage{amscd}
\usepackage{amsthm}
\usepackage{amsfonts}
\usepackage{amssymb}
\usepackage{eucal}
\usetikzlibrary{arrows}

\usepackage{latexsym,todonotes}

\usepackage[linktocpage=true]{hyperref}
\hypersetup{colorlinks,linkcolor=blue,urlcolor=orange,citecolor=blue}
\usepackage[all, cmtip]{xy}

\usepackage{stmaryrd}
\usepackage{color} 

\newtheorem{thm}{Theorem} [section]
\newtheorem{cor}[thm]{Corollary}
\newtheorem{lem}[thm]{Lemma}
\newtheorem{prop}[thm]{Proposition}
\newtheorem{conj}[thm]{Conjecture}

\theoremstyle{definition}
\newtheorem{definition}[thm]{Definition}
\newtheorem{example}[thm]{Example}

\theoremstyle{remark}
\newtheorem{rem}[thm]{Remark}

\numberwithin{equation}{section}
\begin{document}

\newcommand{\thmref}[1]{Theorem~\ref{#1}}
\newcommand{\secref}[1]{Section~\ref{#1}}
\newcommand{\lemref}[1]{Lemma~\ref{#1}}
\newcommand{\propref}[1]{Proposition~\ref{#1}}
\newcommand{\corref}[1]{Corollary~\ref{#1}}
\newcommand{\remref}[1]{Remark~\ref{#1}}
\newcommand{\eqnref}[1]{(\ref{#1})}

\newcommand{\exref}[1]{Example~\ref{#1}}

\newtheorem{innercustomthm}{{\bf Theorem}}
\newenvironment{customthm}[1]
  {\renewcommand\theinnercustomthm{#1}\innercustomthm}
  {\endinnercustomthm}
  
  \newtheorem{innercustomcor}{{\bf Corollary}}
\newenvironment{customcor}[1]
  {\renewcommand\theinnercustomcor{#1}\innercustomcor}
  {\endinnercustomthm}
  
  \newtheorem{innercustomprop}{{\bf Proposition}}
\newenvironment{customprop}[1]
  {\renewcommand\theinnercustomprop{#1}\innercustomprop}
  {\endinnercustomthm}

\newcommand{\nc}{\newcommand}
 \nc{\A}{\mathcal A}
\nc{\Ainv}{\A^{\rm inv}}
\nc{\aA}{{}_\A}
\nc{\aAp}{{}_\A'}
\nc{\aff}{{}_\A\f}
\nc{\aL}{{}_\A L}
\nc{\aM}{{}_\A M}
\nc{\Bin}{B_i^{(n)}}
\nc{\dL}{{}^\omega L}
\nc{\Z}{{\mathbb Z}}
 \nc{\C}{{\mathbb C}}
 \nc{\N}{{\mathbb N}}
 \nc{\fZ}{{\mf Z}}
 \nc{\F}{{\mf F}}
 \nc{\Q}{\mathbb{Q}}
 \nc{\la}{\lambda}
 \nc{\ep}{\epsilon}
 \nc{\h}{\mathfrak h}
 \nc{\He}{\bold{H}}
 \nc{\htt}{\text{tr }}
 \nc{\n}{\mf n}
 \nc{\g}{{\mathfrak g}}
 \nc{\DG}{\widetilde{\mathfrak g}}
 \nc{\SG}{\breve{\mathfrak g}}
 \nc{\is}{{\mathbf i}}
 \nc{\V}{\mf V}
 \nc{\bi}{\bibitem}
 \nc{\E}{\mc E}
 \nc{\ba}{\tilde{\pa}}
 \nc{\half}{\frac{1}{2}}
 \nc{\hgt}{\text{ht}}
 \nc{\ka}{\kappa}
 \nc{\mc}{\mathcal}
 \nc{\mf}{\mathfrak} 
 \nc{\hf}{\frac{1}{2}}
\nc{\ov}{\overline}
\nc{\ul}{\underline}
\nc{\I}{\mathbb{I}}
\nc{\xx}{{\mf x}}
\nc{\id}{\text{id}}
\nc{\one}{\bold{1}}
\nc{\Qq}{\Q(q)}
\nc{\ua}{\mf{u}}
\nc{\nb}{u}
\nc{\inv}{\theta}
\nc{\mA}{\mathcal{A}}
\newcommand{\TT}{\mathbf T}
\newcommand{\TA}{{}_\A{\TT}}
\newcommand{\tK}{\widetilde{K}}

\nc{\U}{\bold{U}}
\nc{\Udot}{\dot{\U}}
\nc{\f}{\bold{f}}
\nc{\fprime}{\bold{'f}}
\nc{\B}{\bold{B}}
\nc{\Bdot}{\dot{\B}}
\nc{\Dupsilon}{\Upsilon^{\vartriangle}}
\newcommand{\T}{\texttt T}
\newcommand{\vs}{\varsigma}
\newcommand{\Pa}{{\bf{P}}}
\newcommand{\Padot}{\dot{\bf{P}}}

\nc{\ipsi}{\psi_{\imath}}
\nc{\Ui}{{\bold{U}^{\imath}}}
\nc{\Uidot}{\dot{\bold{U}}^{\imath}}
 \nc{\be}{e}
 \nc{\bff}{f}
 \nc{\bk}{k}
 \nc{\bt}{t}
 \nc{\BLambda}{{\Lambda_{\inv}}}
\nc{\Ktilde}{\widetilde{K}}
\nc{\bktilde}{\widetilde{k}}
\nc{\Yi}{Y^{w_0}}
\nc{\bunlambda}{\Lambda^\imath}
\newcommand{\Iwhite}{\I_{\circ}}
\nc{\ile}{\le_\imath}
\nc{\il}{<_{\imath}}

\newcommand{\ff}{B}

\nc{\qq}{(q_i^{-1}-q_i)}
\nc{\qqq}{(1-q_i^{-2})^{-1}}
\nc{\qqqj}{(1-q_j^{-2})^{-1}}

\nc{\etab}{\eta^{\bullet}}
\newcommand{\Iblack}{\I_{\bullet}}
\newcommand{\wb}{w_\bullet}
\newcommand{\UIblack}{\U_{\Iblack}}

\newcommand{\blue}[1]{{\color{blue}#1}}
\newcommand{\red}[1]{{\color{red}#1}}
\newcommand{\green}[1]{{\color{green}#1}}
\newcommand{\white}[1]{{\color{white}#1}}

\newcommand{\huanchentodo}{\todo[inline,color=orange!20, caption={}]}
\newcommand{\wtodo}{\todo[inline,color=green!20, caption={}]}

\newcommand{\dvd}[1]{t_{\odd}^{{(#1)}}}
\newcommand{\dvp}[1]{t_{\ev}^{{(#1)}}}
\newcommand{\ev}{\mathrm{ev}}
\newcommand{\odd}{\mathrm{odd}}

\title[Canonical bases arising from quantum symmetric pairs]
{Canonical bases arising from quantum \\
symmetric pairs of Kac-Moody type}
 
 \author[Huanchen Bao]{Huanchen Bao}
\address{Department of Mathematics, National University of Singapore, Singapore}
\email{huanchen@nus.edu.sg}

\author[Weiqiang Wang]{Weiqiang Wang}
\address{Department of Mathematics, University of Virginia, Charlottesville, VA 22904}
\email{ww9c@virginia.edu}

\begin{abstract}  
For quantum symmetric pairs $(\bold{U}, \bold{U}^\imath)$ of Kac-Moody type, we construct $\imath$canonical bases for the highest weight integrable $\bold U$-modules and their tensor products regarded as $\bold{U}^\imath$-modules, as well as an $\imath$canonical basis for the modified form of the $\imath$quantum group $\bold{U}^\imath$. A key new ingredient is a family of explicit elements called $\imath$divided powers, which are shown to generate the integral form of $\dot{\bold{U}}^\imath$. We prove a conjecture of Balagovic-Kolb, removing a major technical assumption in the theory of quantum symmetric pairs. Even for quantum symmetric pairs of finite type, our new approach simplifies and strengthens the integrality of quasi-K-matrix and the constructions of $\imath$canonical bases, by avoiding a case-by-case rank one analysis and removing the strong constraints on the parameters in a previous work. 
\end{abstract}

\maketitle

\let\thefootnote\relax\footnotetext{{\em 2020 Mathematics Subject Classification.} Primary 17B10.}
\let\thefootnote\relax\footnotetext{{\em Keywords:} quantum symmetric pairs, canonical bases, Kac-Moody type.}
\setcounter{tocdepth}{1}
\tableofcontents

\section{Introduction}

\subsection{Background}
\label{subsec:history}

Drinfeld-Jimbo quantum groups $\U =\U(\mathfrak g)$ are deformations of universal enveloping algebras of simple or symmetrizable Kac-Moody Lie algebras $\mathfrak g$, associated with a generalized Cartan matrix $(a_{ij})_{i\in \I}$. The theory of canonical bases for quantum groups has been developed by Lusztig and also by Kashiwara \cite{Lu90, Ka91, Lu91, Lu92, Lu94, Ka94}. It has applications to Kazhdan-Lusztig theory and geometric representation theory; it was also a major motivation for several exciting active research directions in the past decade including categorification and cluster algebras. 


The classification of real simple Lie algebras, or equivalently, the symmetric pairs $(\mathfrak g, \mathfrak g^\theta)$, was achieved by \'E. Cartan, and they are in bijection with (bicolored) Satake diagrams, $\I =\I_\circ \cup \Iblack$; cf. \cite{Ar62}. As a deformation of symmetric pairs, the quantum symmetric pair (QSP) $(\U, \Ui)$ was formulated by Letzter \cite{Le99, Le02} in finite type with the Satake diagrams as inputs; this construction has been generalized by Kolb \cite{Ko14} to the Kac-Moody setting. Here $\Ui$ is a coideal subalgebra of $\U$ and will be referred to as an $\imath$quantum group on its own. A subtle and deep feature of QSP is that $\Ui =\U^{\imath}_{\vs,\kappa}$ depends on a multiple of parameters $\vs =(\vs_i)_{i\in \I_\circ}$ and $\kappa =(\kappa_i)_{i\in \I_\circ}$ subject to some constraints; specializing at $q=1$ and $\vs_i=1$ allows one to recover the symmetric pairs. 

In recent years, it is getting increasingly clear that a number of fundamental constructions for quantum groups admit highly nontrivial generalizations to the setting of QSP. In \cite{BW18b} (also cf. \cite{BW18a}), the authors developed a theory of $\imath$canonical bases for the QSP $(\U, \Ui)$ {\em of finite type}, on $\Ui$-modules as well as a modified form $\Uidot$. The constructions of quasi-$K$-matrix and universal $K$-matrix for QSP \cite{BW18a, BK19} have played a significant role too. 
The $\imath$quantum group $\Ui$ of quasi-split type AIII and its $\imath$canonical basis have applications to super Kazhdan-Lusztig theory of type BCD \cite{BW18a, Bao17}, admit a geometric realization \cite{BKLW, LW18} and a KLR type categorification \cite{BSWW}. The $\imath$quiver varieties introduced more recently by Y.~Li \cite{Li19}  have intimate connections with QSPs.

\subsection{Obstacles}
Let us expand on some of the main obstacles toward canonical bases arising from QSP before the current work. Set $\mA =\Z[q,q^{-1}]$. 
Similar to the canonical bases for quantum groups, the $\imath$canonical bases on modules require 3 ingredients: 
\begin{enumerate}
\item
a bar involution $\ipsi$;
\item
an integral $\mA$-form; 
\item
a $\Z[q^{-1}]$-lattice (with some distinguished basis). 
\end{enumerate}

The bar involution on $\Ui$ in general was proposed in \cite{BW18a} and then constructed in great generality by Balagovic-Kolb in \cite{BK15} who also specified the constraints on the parameters $\vs$ and $\kappa$. On the other hand, the new bar involution $\ipsi$ at the level of $\U$-modules (regarded as $\Ui$-modules by restriction) requires the notion of quasi-$K$-matrix due to the authors \cite{BW18a}, which is a QSP generalization of Lusztig's quasi-$R$-matrix; the existence of quasi-$K$-matrix in much generality has been subsequently established by \cite{BK19}. 

These general constructions in \cite{BK15, BK19} largely rely on two crucial assumptions. The first (and somewhat mild) assumption is that the Cartan integers are bounded, i.e., $|a_{ij}| \le 3$. The construction of a bar involution on $\Ui$ relies on explicit Serre type presentations of $\Ui$, which are only available under the bound assumption on $(a_{ij})$ until very recently. 
In a recent work \cite{CLW18}, a Serre presentation was obtained for the quasi-split $\imath$quantum group $\Ui$ (i.e. associated to the Satake diagram with $\Iblack =\emptyset$) without any bound constraint on $(a_{ij})$, and the existence of the bar involution on such $\Ui$ follows. Note the QSP of finite and affine types are all covered. 

The second (and a more serious) assumption is the validity of a conjecture in \cite{BK15} that certain numbers $\nu_i \in \{1, -1\}$ for $i\in \I_\circ$ should always be $\nu_i=1$; for a precise formulation see \eqref{eq:Zi}--\eqref{eq:nu} and Conjecture~\ref{conj:BK}. It was shown in \cite{BK15} that $\nu_i=1$ always holds for $\U$ of finite type. For $\U$ of Kac-Moody type, the conjecture is only known for $\U$ of locally finite type (cf. \cite[\S5.4]{BW18b}), and in this case it follows from the counterpart in the finite type since the definition of $\nu_i$ is local. We shall refer to this BK conjecture as the {\em Fundamental Lemma of QSP}, since while this sounds like a mere technicality but several major constructions in great generality depend on it. For example, the quasi-$K$-matrix as well as universal $K$-matrix in \cite{BK19} (which was first constructed in \cite[\S2.5]{BW18a} for quasi-split QSP of type AIII) both rely on the validity of the Fundamental Lemma of QSP. 

The approach developed in \cite{BW18a} on the integrality issue for the quasi-$K$-matrix $\Upsilon$ and the modified form $\Uidot$ of finite type is rather tedious and technical as it is based on a case-by-case analysis in the 8 real rank one finite type cases. Such an approach toward the integrality is not generalizable to $\Ui$ of Kac-Moody type, as it is probably impossible to classify the real rank one $\imath$quantum groups of Kac-Moody type (there is already a zoo of real rank one $\imath$quantum groups of affine type; cf. \cite{RV16}). 

\subsection{Goal}

The goal of this paper is to develop a general theory of $\imath$canonical bases for the QSP of {\em Kac-Moody type}. We shall construct $\imath$canonical bases on integrable highest weight $\U$-modules and their tensor products. We shall also construct the $\imath$canonical basis on the modified form $\Uidot$. (Note that the canonical basis on $\Udot$ \cite{Lu92} can be viewed as the $\imath$canonical basis for the $\imath$quantum group for the QSP of diagonal type.) Our goal is achieved by building on the foundational works \cite{Lu94, BK19}, following closely several constructions in \cite{BW18a} when applicable (including a projective system of based $\Ui$-modules), and more crucially, introducing several new ideas to overcome the major obstacles as mentioned above.

Even for QSP of finite type, the results in this paper here strengthen the main results in \cite{BW18b} by allowing general integral parameters $\vs_i$ and simplifying the approach {\em loc. cit.} by substituting the tedious case-by-case real rank one analysis therein with a conceptual $\imath$divided powers construction. 

We further provide a proof of the Fundamental Lemma of QSP in full generality. This removes a major technical assumption in \cite{BK15, BK19} toward the existence of bar involution, quasi-$K$-matrix and universal $K$-matrix, and as well as for the constructions of $\imath$canonical bases in this paper.

\subsection{Fundamental Lemma of QSP}
The Fundamental Lemma of QSP is actually a statement for quantum groups, motivated by the study of quantum symmetric pairs. By \cite[Proposition~2.5]{BK15} $\sigma\tau (Z_i)= \nu_i Z_i$, where $\nu_i \in\{1, -1\}$, $Z_i$ is defined in \eqref{eq:Zi} and $\sigma$ is the anti-involution of $\U$ which fixes $E_i, F_i$ for all $i\in \I$, and $\tau$ is a diagram automorphism. We show that $Z_i$ descends to a certain cell quotient of $\Udot$ as a nonzero scalar multiple of a canonical basis element, and it follows that $\nu_i=1$ as the canonical basis is preserved by $\sigma$ and $\tau$. Our proof relies in an essential way on Lusztig's theory of based modules and cells on quantum groups \cite[Chapters~ 27, 29]{Lu94}. Some old results of Joseph and Letzter \cite{JL94, JL96} also play a role. 

\subsection{The $\imath$divided powers}

A key new construction in this paper is a family of explicit elements called $\imath$divided powers, for each $i \in \I$, which by definition satisfy the 3 properties: bar invariant, integral and having a leading term the standard divided powers embedded as elements in $\Udot$. The $\imath$divided powers associated to $i \in \Iblack$ are the standard divided powers $F_i^{(n)}, E_i^{(n)}$ as in \cite{Lu94}. The $\imath$divided powers associated to $i \in \I_\circ$ with $\tau i\neq i$  or to $i \in \I_\circ$ with $\wb i =i =\tau i$ were introduced and studied earlier in \cite{BW18a, BW18b, BeW18}. For the class of $i \in \I_\circ$ with $\wb i \neq i =\tau i$, our formula for the $\imath$divided powers, denoted by $B_{i,\zeta}^{(n)}$ for $n\ge 1$, is explicit and universal. (Note that 5 out of 8 $\imath$quantum groups of real rank one in finite type belong to this class; cf. \cite[\S3.2, Table~1]{BW18b}.)

The $\imath$divided powers $B^{(n)}_{i, \zeta}$ should be regarded as a leading term of a corresponding $\imath$canonical basis element. 
We eventually show that the $\A$-form ${}_\mA\Uidot$, which is defined in a more conceptual way, is actually generated by the $\imath$divided powers, and ${}_\mA\Uidot$ is a free $\A$-submodule of $\Uidot$ such that $\Uidot =\Q(q)\otimes_{\A} {}_\A\Uidot$.

\subsection{Integrality of $\Upsilon$ in action}

It is neither expected nor needed that the quasi-$K$-matrix for $\Ui$ beyond finite type is integral on its own, based on the knowledge from quasi-$R$-matrix for a quantum group $\U$ beyond finite type (cf. \cite{BW16}). 
With the help of the aforementioned $\imath$divided powers, we show that $\Upsilon$ preserves the integral $\A$-forms on integrable highest weight $\U$-modules and their tensor products. The proof is in part inspired by our argument in \cite{BW16} that Lusztig's quasi-$R$-matrix preserves the integral $\A$-forms of these $\U$-modules; actually the approach developed in \cite{BW16} toward canonical bases on tensor product modules can in turn be viewed as dealing with the special case of QSP of diagonal type. 

For QSP of finite type,  our new approach via the $\imath$divided powers allow us to establish the integrality of $\Upsilon$ for general integral parameters $\vs_i$ (in contrast to $\vs_i \in \pm q^\Z$ in \cite{BW18b}) and to bypass the tedious case-by-case real rank one analysis in \cite[Appendix~A.4--A.7]{BW18a}. (The complete detail takes 23 pages and can be found in Appendix~ A.4--A.11 in the arXiv version~1 of the paper.) 

Following \cite{BW18a, BW18b}, we define a new bar involution $\ipsi =\Upsilon \circ \psi$ on the based $\U$-modules such as integrable highest weight $\U$-modules and their tensor products. As $\ipsi$ preserves the integral $\A$-forms, we are able to construct the $\imath$canonical bases on these modules using their canonical bases from \cite{BW16} (cf. \cite[Part IV]{Lu94}). The $\imath$canonical basis spans the same $\Z[q^{-1}]$-lattices as the usual canonical basis --- this is the characterization property~(3) for $\imath$canonical basis in \S\ref{subsec:history}. 

We further extend the construction of based $\Ui$-modules to tensor products of a based $\Ui$-module with a based $\U$-module, using a construction of $\Theta^\imath$ in \cite{BW18a} and \cite{Ko20}; this was recently carried out for QSP of finite type in \cite{BWW20}. 

\subsection{$\imath$Canonical bases on based modules and $\Uidot$}

Note we have established the $\imath$canonical bases on modules for $\Ui$ with general parameters $\vs_i \in \A$ (in contrast to the very strict constraint that $\vs_i \in \pm q^\Z$ in \cite{BW18b}). Recall the condition on $\vs_i$ in \cite{BW18b} was imposed to ensure that an anti-involution $\wp$ on $\U$  (see Proposition~\ref{prop:invol}) restricts to an involution to the subalgebra $\Ui$ (see \cite[Proposition~4.6]{BW18b}). 

We make a crucial observation here that the general construction of $\imath$canonical basis on the modified form $\Uidot$ does not rely on the fact that $\wp$ preserves $\Ui$ (this was rather a mental block for us for quite some time). In all our constructions toward the $\imath$canonical basis on $\Uidot$, we only need to use the twist of $\wp$ on $\U$-modules. Recall that any $\U$-module is automatically a $\Ui$-module by restriction. 

The construction for the canonical basis on the modified $\imath$quantum group $\Uidot$ relies crucially on the based $\U$-submodules $L(w\lambda, \mu) \subset L(\lambda) \otimes L(\mu)$, for a Weyl group element $w$; cf. \eqref{eq:Lwlamu}. 
In finite type, we proved that $L(w\lambda, \mu)$ is a based module using an identification $L(\lambda) \otimes L(\mu) \cong  {}^\omega L( -w_0\lambda) \otimes L(\mu)$, where $w_0$ is the longest element in the Weyl group. 
Instead, Kashiwara's theory of extreme weight modules \cite{Ka94} is used to prove that $L(w\lambda, \mu)$ is a based submodule of $L(\lambda) \otimes L(\mu)$ over $\U$ of Kac-Moody type; see Theorem~\ref{thm:based-Li}. 	
	
With all these preparations and observations, we follow \cite{BW18a} closely to build a projective system of based $\Ui$-modules (which generalizes Lusztig's construction \cite[Part IV]{Lu94}), and establish the $\imath$canonical basis of $\Uidot$ (and of ${}_\A\Uidot$). 

For QSP of affine type AIII, a geometric realization of $\Uidot$ and its $\imath$canonical bases was first given in \cite{FL+20}. 
\subsection{Applications} 

For the existence of bar involution, we shall make the basic assumption that $|a_{ij}| \le 3$ or $\Iblack =\emptyset$ throughout the paper. 
It seems possible that the new $\imath$divided powers introduced in this paper might help to obtain a Serre presentation for $\Ui$ (and then a bar involution) by weakening or removing this bound assumption on the generalized Cartan matrices in the long run. This is exactly how the (old) $\imath$divided powers helped solving this very problem for the quasi-split  $\imath$quantum groups $\Ui$ in \cite{CLW18}. 

The $\imath$divided powers are expected to play a key role in computational aspects of $\imath$canonical bases and related combinatorics. 
They will help to shed new light on the categorification and geometric realization of QSP, and in addition have applications in the study of quantum symmetric pairs at roots of 1. 

In \cite{RV16}, several crucial constructions for QSP were carried over for a more general class of quantum algebras associated to generalized Satake diagrams. It is interesting to explore if the technique of $\imath$divided powers introduced in this paper allows a possible further generalization of integral forms and $\imath$canonical basis in this generalized QSP setting.

\subsection{The organization} 

The paper is organized as follows. Sections~\ref{sec:lem}-\ref{sec:iCBM} contain the main new ideas of this paper.

In Section~\ref{sec:QG}, we review and set up notations for quantum groups and canonical basis. The new result in this section is the existence of a family of based $\U$-modules $L(w\la, \mu)$. 

In Section~\ref{sec:Ui}, we set up notations for quantum symmetric pairs $(\U, \Ui)$. We recall some earlier work \cite{BK19} and summarize various results in \cite{BW18b} which remain valid in the Kac-Moody setting. The existence of an anti-involution $\sigma_\imath$ on $\Ui$ (analogous to an anti-involution $\sigma$ on $\U$) is new. We assume in Section~\ref{sec:Ui} the validity of the Fundamental Lemma of QSP, so the results in this section work in great generality. 
The Fundamental Lemma of QSP, which is a conjecture of Balagovic and Kolb, is then proved in Section~\ref{sec:lem}. 

In Section~\ref{sec:iDP}, we define explicitly and study in depth the $\imath$divided powers in 3 separate classes, including a major new class of $i\in \I_\circ$ with $\wb i \neq i=\tau i$. We prove that the $\imath$divided powers are integral, bar invariant with suitable leading term;  we shall see in \S\ref{sec:iCBU} that they generate the $\A$-form ${}_\A \Uidot$.

In Section~\ref{sec:iCBM}, we show that the quasi-$K$-matrix $\Upsilon$ preserves the $\A$-forms of highest weight integrable $\U$-modules and their tensor products, and then construct $\imath$canonical bases on these modules. We further construct $\imath$canonical bases on a tensor product of a based $\Ui$-module and a based $\U$-module. 

In Section~\ref{sec:iCBU}, we establish the $\imath$canonical basis on $\Uidot$, and show that the $\A$-subalgebra ${}_\A \Uidot$ of $\Uidot$ defined in Definition~\ref{def:mAUidot} is indeed a free $\A$-module such that $\Q(q) \otimes_{\A}  {}_\A \Uidot = \Uidot$. This is essentially a summary of \cite[\S6]{BW18b}, now in the Kac-Moody setting, based on the results in previous sections. A  different (and more elementary) partial order inspired by \cite{BWW20} is used here.

\subsection*{Acknowledgments.}
This project cannot be completed without generous helps from many people over the years. We are indebted to Masaki Kashiwara for supplying a proof of Theorem~\ref{thm:based-Li}. We thank Gail Letzter for her expertise and help regarding Lemma~\ref{lem:nonzero}, and thank Stefan Kolb for helpful clarifications of his work with Balagovic. We thank Hideya Watanabe for enlightening discussions regarding the modified $\imath$quantum groups. 
We thank Collin Berman (supported by REU in WW's NSF grant) for various computer explorations on $\imath$divided powers, and Thomas Sale for a verification of a crucial example. We would also like to thanks an anonymous referee for helpful comments.

HB is supported by NUS-MOE grant R-146-000-294-133 and R-146-001-294-133. WW is partially supported by NSF grant DMS-1702254 and DMS-2001351.

\section{Quantum groups and canonical bases}
  \label{sec:QG}

In this section, we review some basic constructions and set up notations in quantum groups. We then establish a family of based $\U$-modules $L(w\la, \mu)$.

\subsection{}
Let $q$ be an indeterminate. Consider a free $\Qq$-algebra $'\f$ generated by $\theta_{i}$ for ${i \in \I}$ associated with the Cartan datum of type $(\I, \cdot)$. As a $\Qq$-vector space, $'\f$ has a weight space decomposition as $'\f = \bigoplus_{\mu \in {\N}[\I]}~ '\f_{\mu},$ where $\theta_{i}$ has weight $i$ for all $i \in \I$. For $\mu =\sum_{i\in\I} a_i i$, the height of $\mu$ is denoted by $\hgt (\mu) =\sum_{i\in \I} a_i$.
For any $x \in \fprime_\mu$, we set $|x| = \mu$. For any $i \in \I$, we set $q_{i} = q^{\frac{i \cdot i}{2}}$. Let $W$ be the corresponding Weyl group generated by simple reflections $s_i$ for $i \in \I$.

For each $i \in \I$, we define $r_{i}, {}_i r: \fprime \rightarrow \fprime$ to be the unique $\Q(q)$-linear maps  such that 
\begin{align}  \label{eq:rr}
\begin{split}
r_{i}(1) = 0, \quad r_{i}(\theta_{j}) = \delta_{ij},
\quad r_{i}(xx') = xr_{i}(x') + q^{i \cdot \mu'}r_{i}(x)x',
 \\
{}_{i}r(1) = 0, \quad {}_{i}r(\theta_{j}) = \delta_{ij},
\quad {}_{i}r(xx') =q^{i \cdot \mu }x \,_{i}r(x') +{ _{i}r(x)x'},
\end{split}
\end{align}
for all $x \in \fprime_{\mu}$ and $x' \in \fprime_{\mu'}$. 

%

Let $(\cdot, \cdot)$ be the symmetric bilinear form on $\fprime$ defined in  \cite[1.2.3]{Lu94}.
Let ${\bf I}$ be the radical of the symmetric bilinear form $(\cdot,\cdot)$ on $\fprime$.
 For $i\in \I, n \in \Z$ and $s \in \N$, we define 
\[
[n]_i = \frac{q_i^n-q_i^{-n}}{q_i - q^{-1}_i} \quad  \text { and } \quad [s]^!_i = \prod^s_{j=1} [j]_i.
\] 
We shall also use the notation 
\[
\begin{bmatrix}
n\\
s
\end{bmatrix}_i
=
\frac{[n]^!_i}{[s]^!_i [n-s]^!_i}, \quad \text{ for } 0 \le s \le n.
\]
It is known \cite{Lu94} that ${\bf I}$ is 
generated by the quantum Serre relators $S(\theta_i, \theta_j)$, for $ i \neq  j \in \I$, where
\begin{equation}  \label{eq:Sij}
S(\theta_i, \theta_j) =  \sum^{1-a_{ij}}_{s=0}(-1)^{s} \begin{bmatrix}1-a_{ij}\\s \end{bmatrix}_i\theta_{i}^s \theta_{j} \theta_{i}^{1-a_{ij}-s}. 
\end{equation}
Let $\f =\fprime/\bf{I}$.
We introduce the divided power $\theta^{(a)}_{i} = \theta^a_{i}/[a]_i^!$ for  $a \ge 0$. Let 
\[
\mA =\Z[q,q^{-1}].
\]
Let $_\mA\f$ be the $\mA$-subalgebra of $\f$ 
generated by $\theta^{(a)}_{i}$ for various $a \ge 0$ and $i \in \I$.

\subsection{}
Let $(Y,X, \langle\cdot,\cdot\rangle, \cdots)$ be a root datum of type $(\I, \cdot)$; cf. \cite{Lu94}. We define a partial order $\leq$ on the weight lattice $X$ as follows: for $\la, \la' \in X$, 
 \begin{equation}
   \label{eq:leq}
  \lambda \le \lambda' \text{ if and only if } \lambda' -\lambda \in \N[\I]. 
 \end{equation}
 
 %
The quantum group $\U$ associated with this root datum  $(Y,X, \langle\cdot,\cdot\rangle, \cdots)$ is the associative $\Qq$-algebra generated by $E_{i}$, $F_{i}$ for $i \in \I$ and $K_{\mu}$ for $\mu \in Y$, subject to the following relations: for all $\mu, \mu' \in Y$ and $i \neq j \in \I$,
\begin{align*}
K_{0} =1, \qquad &K_{\mu} K_{\mu'} = K_{\mu +\mu'},  
  \\
 K_{\mu} E_{j} = q^{\langle \mu, j' \rangle} E_{j} K_{\mu}, &  \qquad 
 K_{\mu} F_{j} = q^{-\langle \mu, j' \rangle} F_{j} K_{\mu}  , \\
 E_{i} F_{j} -F_{j} E_{i} &= \delta_{i,j} \frac{\Ktilde_{i} -\Ktilde_{-i}}{q_i-q_i^{-1}}, \\
S(F_{i}, F_{j}) &= S(E_{i}, E_{j}) = 0,  
\end{align*}
where $\Ktilde_{\pm i} = K_{\pm \frac{i \cdot i}{2} i}$ and $S(\cdot, \cdot)$ is defined as \eqref{eq:Sij}.

Let $\U^+$, $\U^0$ and $\U^-$ be the $\Qq$-subalgebra of $\U$ generated by $E_{i} (i \in \I)$, $K_{\mu} (\mu \in Y)$, 
and $F_{i}  (i \in \I)$  respectively.
We identify $\f \cong \U^-$ by matching the 
generators $\theta_{i}$ with $F_{i}$. This identification induces
a bilinear form $(\cdot, \cdot)$ on $\U^{-}$ and $\Qq$-linear maps $r_i, {}_i r$ $(i\in \I)$ on $\U^-$.
Under this identification, we let $\U_{-\mu}^-$ be the image of $\f_\mu$.
Similarly we have $\f \cong \U^+$ by identifying $\theta_{i}$ with $E_{i}$. 
We let $_\mA\U^-$ (respectively, $_\mA\U^+$) denote the image of $_\mA\f$ under this isomorphism, 
which is generated by all divided powers $F^{(a)}_{i} =F_{i}^a/[a]_i^!$ (respectively, $E^{(a)}_{i} =E_{i}^a/[a]_i^!$). The coproduct $\Delta: \U \rightarrow \U \otimes \U$ is defined as follows:
\begin{equation}\label{eq:Delta}
\Delta(E_i)  = E_i \otimes 1 + \widetilde{K}_i \otimes E_i, \quad \Delta(F_i) = 1 \otimes F_i + F_i \otimes \widetilde{K}_{-i}, \quad \Delta(K_{\mu}) = K_{\mu} \otimes K_{\mu}.
\end{equation}

\subsection{} 

We recall several symmetries of $\U$; cf. \cite{Lu94}.
\begin{prop} \label{prop:invol}
{\quad}
\begin{enumerate}
\item There is an  involution $\omega$ of the $\Qq$-algebra $\U$ such that $\omega(E_{i}) =F_{i}$, $\omega(F_{i}) =E_{i}$,  and $\omega(K_{\mu}) = K_{-\mu}$ for all $i \in \I$ and $\mu \in Y$.

\item There is an anti-involution $\wp$ of the $\Qq$-algebra $\U$ such that 
$\wp(E_{i}) = q^{-1}_i F_{i} \Ktilde_{i}$, $\wp(F_{i}) = q^{-1}_i E_{i} \Ktilde_i^{-1}$  and $\wp(K_{\mu}) = K_{\mu}$ for all $i \in \I$ and $\mu \in Y$.

\item	There is an anti-involution $\sigma$ of the $\Qq$-algebra $\U$ such that 
$\sigma(E_{i}) = E_{i}$, $\sigma(F_{i}) = F_{i}$  and $\sigma(K_{\mu}) = K_{-\mu}$ for all $i \in \I$ and $\mu \in Y$.

\item There is a bar involution $\overline{\phantom{x}}$ of the $\Q$-algebra $\U$ such that 
$q \mapsto q^{-1}$, $\ov{E}_{i}= E_{i}$, $\ov{F}_{i}=F_{i}$,  and $\ov{K}_{\mu}=K_{-\mu}$ for all $i \in \I$ and $\mu \in Y$.
(Sometimes we denote the bar involution on $\U$ by $\psi$.)

\item There are automorphisms $\T''_{i, -e}$ (for $e=\pm 1$, $i\in I$) of the $\Qq$-algebra $\U$ such that 
\begin{equation*}
\begin{split}
&\T''_{i, -e} (E_i) = -F_i \widetilde{K}_{-ei} , \quad \T''_{i, -e} (F_i) = -  \widetilde{K}_{ei}E_i,  \quad \T''_{i, -e} (K_\mu) = K_{s_i(\mu)}; \\
&\T''_{i, -e} (E_j) =  \sum_{r+s = - \langle i, j'\rangle} (-1)^r q^{er}_{i} E^{(s)}_i E_j E^{(r)}_i \quad \text{ for } j\neq i. 
\end{split}
\end{equation*}
\end{enumerate}
\end{prop}
Since $\T''_{i, +1}$ satisfies the braid group relation, we can define the automorphism $\T''_{w, +1}$ of $\U$, associated to $w\in W$, in a standard fashion. To simplify the notation, throughout the paper we shall often write 
\[
\T_i = \T''_{i, +1}, \quad \text{ and } \quad \T_w = \T''_{w, +1}, \text{ for } w  \in W.
\]

\subsection{}
Let $M(\lambda)$ be the Verma module of $\U$ with highest weight $\lambda\in X$ and
with a highest weight vector denoted by $\eta_{\lambda}$.  
We define a lowest weight $\U$-module $^\omega M(\lambda)$, which 
has the same underlying vector space as $M(\lambda)$ but 
with the action twisted by the involution $\omega$ given in Proposition~\ref{prop:invol}.
When considering $\eta_{\lambda}$ as a vector in $^\omega M(\lambda)$, 
we shall denote it by $\xi_{-\lambda}$. 

Let 
$$X^+ = \big\{\lambda \in X \mid \langle i, \lambda \rangle \in {\N}, \forall i \in \I \big \}$$ 
 be the set of dominant integral weights. 
 By $\la \gg 0$ we shall mean that the integers $\langle i, \lambda \rangle$ for all $i$ are sufficiently large. The Verma module $M(\lambda)$ associated to $\la \in X$ has a unique simple quotient $\U$-module, denoted by $L(\lambda)$. We shall abuse the notation and denote by $\eta_\lambda \in L(\lambda)$ the image of the highest weight vector $\eta_\lambda \in M(\lambda)$. Similarly we define the $\U$-module $^\omega L(\lambda)$ of lowest weight $-\la$ with lowest weight vector $\xi_{-\la}$.  
 For $\la \in X^+$, we let $_\mA L(\lambda) ={_\mA\U^-}  \eta_\la$ and $^\omega _\mA L(\lambda) ={_\mA\U^+ } \xi_{-\la}$ be the $\mA$-submodules of $L(\lambda)$ and $^\omega L(\lambda) $, respectively. 

%
There is a canonical basis $\B$ on $\f$, a canonical basis $\{b^+|b\in \B\}$ on $\U^+$, and a canonical basis $\{b^-|b\in \B\}$ on $\U^-$. 
For each $\la\in X^+$, there is a subset $\B(\la)$ of $\B$ so that $\{b^- \eta_\la |b\in \B(\la)\}$ (respectively, $\{b^+\xi_{-\la} |b\in \B(\la)\}$)  forms a canonical basis of $L(\la)$ (respectively, $\dL(\la)$).  For any Weyl group element $w\in W$, let $\eta_{w\la}$ denote the unique canonical basis element of weight $w\la$. 

Let $\Udot =\oplus_{\zeta \in X} \Udot \one_\zeta$ be the idempotented quantum group and $\aA \Udot$ be its $\A$-form. Then $\Udot$ admits a canonical  basis 
$\dot{\B}  = \{ b_1 \diamondsuit_{\zeta} b_2   \vert (b_1, b_2) \in \B \times \B, \zeta \in X \}$; cf. \cite[Part ~IV]{Lu94}.

\subsection{}
 \label{subsec:P}
For any $\Iblack \subset \I$, let $\U_{\Iblack}$ be the $\Qq$-subalgebra of $\U$ generated by $F_{i} (i \in \Iblack)$, $E_{i} (i \in \Iblack)$ and $K_{i} ( i \in \Iblack)$.  
Let $\B_{\I_\bullet}$ be the canonical basis of $\f_{\Iblack}$ (here $\f_{\Iblack}$ is simply a version of $\f$ associated to $\Iblack$), which induces canonical bases on  $\U^{-}_{\I_\bullet}$ and $\U^{+}_{\I_\bullet}$. Let $\Pa= \Pa_{\I_\bullet}$ be the $\Qq$-subalgebra of $\U$ generated by $\U_{\Iblack}$ and $\U^{-}$. We denote by $L_{\Iblack}(\la)$ the simple $\U_{\Iblack}$-module of highest weight $\la$.

We introduce the following subalgebra of $\Udot$:
$$\Padot = \bigoplus_{\lambda \in X} \Pa \one_{\lambda}.
$$
We further set ${}_\mA \Padot = \Padot \cap {}_\mA \Udot$.
%
%

\subsection{Based submodules $L(w\la, \mu)$}

Recall the theory of based $\U$-modules of Lusztig \cite[Chapter 27]{Lu94} for $\U$ of finite type, which was extended by the authors in \cite{BW16} for $\U$ of Kac-Moody type. For $\lambda, \mu \in X^+$ and $w\in W$, we introduce the following $\U$-submodule:
\begin{align}
  \label{eq:Lwlamu}
L(w\lambda, \mu) = \U (\eta_{w\lambda} \otimes \eta_{\mu}) \subset L(\lambda) \otimes L(\mu).
\end{align}


The following theorem is the main result of this section, whose proof was kindly communicated to us by Kashiwara. 

\begin{thm}
 \label{thm:based-Li}
Let $\lambda, \mu \in X^+$, and $w\in W$. Then the $\U$-submodule $L(w\lambda, \mu)$   
is a based $\U$-submodule of $L(\lambda) \otimes L(\mu)$. 
\end{thm}

\begin{proof} 
Write $w\la =\la_1 -\nu$, for some $\la_1, \nu\in X^+$.
Thanks to \cite[Theorem 2.9, Proposition 2.11]{BW16}, $L(\la_1) \otimes L(\mu)$ is a based $\U$-module, and the map $\chi: L(\la_1 +\mu)\rightarrow L(\la_1) \otimes L(\mu)$, which sends $\eta_{\la_1 +\mu} \mapsto \eta_{\la_1} \otimes \eta_{\mu}$, is a based $\U$-module homomorphism. Therefore,
\[
\chi': =\text{id}_{\;\dL(\nu)} \otimes\chi: \dL(\nu) \otimes L(\la_1 +\mu) \longrightarrow \dL(\nu) \otimes L(\la_1) \otimes L(\mu),  
\]
which sends $\xi_\nu \otimes \eta_{\la_1 +\mu} \mapsto \xi_\nu \otimes \eta_{\la_1} \otimes \eta_{\mu}$, is a based module homomorphism. 

Recall the notion of extremal weight modules \cite[\S8]{Ka94}, and in particular the extremal weight module $L(w\la)$ coincides with $L(\la)$, thanks to $\la \in X^+$. 

It follows from \cite[Proposition~23.3.6]{Lu94} that there exists a $\U$-module homomorphism $\phi: \dL(\nu) \otimes L(\la_1) \rightarrow L(w\la)=L(\la)$, which sends $\xi_\nu \otimes \eta_{\la_1} \mapsto \eta_{w\la}$. By Kashiwara \cite[Lemma~8.2.1]{Ka94}, the map $\phi: \dL(\nu) \otimes L(\la_1) \rightarrow L(w\la)=L(\la)$ is a based module homomorphism. 
Thus, 
\[
\phi' :=\phi \otimes \text{id}_{L(\mu)}: \dL(\nu) \otimes L(\la_1)\otimes L(\mu) \longrightarrow L(\la)\otimes L(\mu), 
\]
which sends $\xi_\nu \otimes \eta_{\la_1} \otimes x \mapsto \eta_{w\la}\otimes x$, for $x\in L(\mu)$, is a based $\U$-module homomorphism. 

Therefore, the composition homomorphism 
\[
\phi' \chi': \dL(\nu) \otimes L(\la_1 +\mu) \longrightarrow L(\la)\otimes L(\mu), 
\]
which sends $\xi_\nu \otimes \eta_{\la_1 +\mu} \mapsto \eta_{w\lambda} \otimes \eta_{\mu}$, is a based $\U$-module homomorphism. 
Since $\xi_\nu \otimes \eta_{\la_1 +\mu}$ is a cyclic vector (i.e., it generates the $\U$-module $\dL(\nu) \otimes L(\la_1 +\mu))$, the $\U$-module $L(w\lambda, \mu)$  is the image of the based module homomorphism $\phi' \chi'$. Hence $L(w\lambda, \mu)$  is a based $\U$-submodule of $L(\la)\otimes L(\mu)$. 
\end{proof}

\begin{rem}
Theorem~\ref{thm:based-Li} for $\U$ of finite type appeared as \cite[Theorem ~2.6]{BW18b} with a different proof.
\end{rem}

\section{The $\imath$quantum groups $\Ui$}
  \label{sec:Ui}
  
In this section, we review the basic definitions and constructions of quantum symmetric pairs, including the bar involution and quasi-$K$-matrix. We also formulate various (old and new) symmetries on $\Ui$ and $\Ui$-modules. Our general setup assumes the validity of the Fundamental Lemma of QSP, which is to be established in Section~\ref{sec:lem}.  

\subsection{}
  \label{subsec:adm}
  
 Let $\tau$ be an involution of the Cartan datum $(\I, \cdot)$; we allow $\tau =\id$. We further assume that $\tau$ extends to  an involution on $X$ and an involution on $Y$, respectively, such that the perfect bilinear pairing is invariant under the involution $\tau$. For any $\lambda \in X$ (or $Y$), we shall write $\lambda^\tau = \tau(\lambda)$.

From now on, we ssume $\I_{\bullet} \subset \I$ is a Cartan subdatum of {\em finite type}. Let $W_{\I_\bullet}$ be the parabolic subgroup of $W$ with $w_{\bullet}$ as  its longest element. 
Let $\rho^\vee_{\bullet}$ be the half sum of all positive coroots in the set $R^{\vee}_{\bullet}$, 
and let $\rho_{\bullet}$ be the half sum of all positive coroots in the set $R _{\bullet}$. 
We shall write 
\begin{equation}
\label{eq:white}
 \I_{\circ} = \I \backslash \I_{\bullet}.
\end{equation}
%

A pair $(\I_{\bullet}, \tau)$ is called {\em admissible} (cf. \cite[Definition~2.3]{Ko14}) if the following conditions are satisfied:
\begin{itemize}
	\item	[(1)]	$\tau (\I_{\bullet}) = \I_{\bullet}$; 
	\item	[(2)]	The action of $\tau$ on $\I_{\bullet}$ coincides with the action of $-w_{\bullet}$; 
	\item	[(3)]	If $j \in \I_{\circ}$ and $\tau(j) = j$, then $\langle \rho^\vee_{\bullet}, j' \rangle \in \Z$.
\end{itemize}
In this paper, all pairs $(\I_{\bullet}, \tau)$ considered are admissible. 

%

Note that
$ \inv =  -w_{\bullet} \circ \tau$
is an involution of $X$ and $Y$. Following \cite{BW18b}, we introduce the $\imath$-weight lattice and $\imath$-root lattice
\begin{align}
  \label{XY}
 \begin{split}
X_{{\imath}} = X /  \breve{X}, & \quad \text{ where } \; \breve{X}  = \{ \la - \inv(\la) \vert \la \in X\},
 \\
Y^{\imath} &= \{\mu \in Y \big \vert \inv(\mu) =\mu \}.
\end{split}
\end{align}
For any $\la \in X$ denote its image in $X_{\imath}$ by $\overline{\la}$.

The involution $\tau$ of $\I$ induces an isomorphism of the $\Q(q)$-algebra $\U$, denoted also by $\tau$,
which sends $E_i \mapsto E_{\tau i}, F_i \mapsto F_{\tau i}$, and $K_\mu \mapsto K_{\tau \mu}$. 

\subsection{}
We recall the definition of quantum symmetric pair $(\U, \Ui)$, where $\Ui$ is a coideal subalgebra of $\U$ from \cite{Le99, Ko14, BK15, BK19}; also cf. \cite[\S3.3]{BW18b}. 

\begin{definition}\label{def:Ui}
The algebra $\Ui$, with parameters 
\begin{equation}
  \label{parameters}
\vs_{i} \in { \Z[q, q^{-1}]}, \quad \kappa_i \in \Z[q,q^{-1}], \qquad \text{ for }  i \in \I_{\circ},
\end{equation}
is the $\Qq$-subalgebra of $\U$ generated by the following elements:
\begin{align*}
F_{i}  &+ \vs_i \T_{w_{\bullet}} (E_{\tau i}) \widetilde{K}^{-1}_i 
+ \kappa_i \widetilde{K}^{-1}_{i} \, (i \in \I_{\circ}), 
 \\
& \quad K_{\mu} \,(\mu \in Y^{\imath}), \quad F_i \,(i \in \I_{\bullet}), \quad E_{i} \,(i \in \I_{\bullet}).
\end{align*}
The parameters are required to satisfy Conditions \eqref{kappa}-\eqref{vs2}: 
\begin{align}
 \label{kappa}
 \begin{split}
\kappa_i &=0 \; \text{ unless } \tau(i) =i, \langle i, j' \rangle = 0 \; \forall j \in \Iblack,
\\
&  
\qquad 
\text{ and } \langle k,i' \rangle \in 2\Z \; \forall k = \tau(k) \in \Iwhite \text{ such that } \langle k, j' \rangle = 0 \; \text{ for all } j \in \Iblack;
\end{split}
\\
\overline{\kappa_i} &= \kappa_i;   \label{kappa2}   
\\
\vs_{i} & =\vs_{{\tau i}} \text{ if }    i \cdot \theta (i) =0;
\label{vs=}
\\
\vs_{{\tau i}} &= (-1)^{ \langle 2\rho^\vee_{\bullet},  i' \rangle } q_i^{-\langle i, 2\rho_{\bullet}+\wb\tau i ' \rangle} {\ov{\vs_{i} }}.   \label{vs2}
\end{align}
\end{definition}
By definition, the algebra $\Ui$ contains $\U_{\Iblack}$ as a subalgebra.

\begin{rem}
\label{rem:par}
Note that the conditions \eqref{kappa}-\eqref{vs2} for the parameters are as general as in \cite{BK19}, except the integral requirement in \eqref{parameters} which is necessary for an integral form of $\Ui$. We are going to develop the theory of $\imath$canonical basis for $\Ui$ in this full generality. This is a significant improvement than the constraint that $\vs_i \in \pm q^{\Z}$ in \cite[Definition~3.5]{BW18b} even in finite type, where \eqref{vs2} was written as $\vs_{{\tau i}} \vs_i = (-1)^{ \langle 2\rho^\vee_{\bullet},  i' \rangle } q_i^{-\langle i, 2\rho_{\bullet}+\wb\tau i ' \rangle}$. 
\end{rem}

\begin{rem}
\label{rem:par2}
Our parameter $\vs_i$ is related to the notations of parameters in \cite{Ko14} and \cite{BK15} via $\vs_i =- s(\tau(i)) c_i$; we shall never need these additional parameters separately. The parameters $\vs_i$ can always be chosen to be $\vs_i \in q^\Z$ by \cite[Remark~ 3.14]{BK15}, once \cite[Conjecture~2.7]{BK15} (i.e., Conjecture~\ref{conj:BK} below) is established in Theorem~\ref{thm:nu=1}. This allows the specialization at $q=1$ of the QSP $(\U, \Ui)$ to the corresponding symmetric pair, justifying the terminology of QSP.
\end{rem}
Set \[
\Ainv = \{f\in \A \mid \bar{f} =f \}.
\]
\begin{rem}
\label{rem:par3}
The coefficient $ (-1)^{ \langle 2\rho^\vee_{\bullet},  i' \rangle } q_i^{-\langle i, 2\rho_{\bullet}+\wb\tau i ' \rangle}$ has been computed explicitly in \cite[Lemma~3.10]{BW18b} in finite type. 
The Satake diagrams of symmetric pairs of real rank one in finite type are listed in \cite[\S3, Table~1]{BW18b}. 
\cite[\S3, Table~ 3]{BW18b} on the values of $\vs_i$ for  quantum symmetric pairs of real rank one is now updated to become the following table, taking into account the relaxed conditions on parameters $\vs_i$ in \eqref{parameters}. 
%
\begin{table}[h]
\begin{tabular}{| c | c | c | c | c | c | c | c |}
\hline
\begin{tikzpicture}[baseline=0]
\node at (0, 0.2) {AI$_1$};
\end{tikzpicture} 
&
\begin{tikzpicture}[baseline=0]
\node at (0, 0.2) {AII$_3$};
\end{tikzpicture}
&
\begin{tikzpicture}[baseline=0]
\node at (0, 0.2) {AIII$_{11}$};
\end{tikzpicture}
	&
\begin{tikzpicture}[baseline=0]
\node at (0, 0.2) {AIV, n$\ge$2};
\end{tikzpicture} 
\\
\hline
$q_1^{-1} \cdot \Ainv$
&
$q \cdot \Ainv$
 &
$\vs_1=\vs_2\in \Ainv$
 &
$\vs_n  = (-1)^{n} q^{n-1} \overline{\vs_1} \in \A$
\\
\hline
\hline
\begin{tikzpicture}[baseline=0]
\node at (0, 0.2) {BII, n$\ge$ 2};
\end{tikzpicture} 
& 
\begin{tikzpicture}[baseline=0]
\node at (0, 0.2) {CII, n$\ge$3};
\end{tikzpicture}  
& 
\begin{tikzpicture}[baseline=0]
\node at (0, 0.2) {DII, n$\ge$4};
\end{tikzpicture}
&
\begin{tikzpicture}[baseline=0]
\node at (0, 0.2) {FII};
\end{tikzpicture}
\\
\hline
$q^{2n-3} \cdot \Ainv$
&
$q^{n-1} \cdot \Ainv$
 &
$q^{n-2} \cdot \Ainv$	
&
$q^5 \cdot \Ainv$
\\
\hline
\end{tabular}
\newline
\end{table}
\end{rem}

It is sometimes convenient to set
\begin{equation}\label{eq:Bi}
\ff_i = 
\begin{cases}
F_{i} + \vs_i  \T_{w_{\bullet}} (E_{\tau i}) \tK^{-1}_i  + \kappa_i \widetilde{K}^{-1}_{i} 
& \text{ if } i \in \I_{\circ};
\\
 F_{i} & \text{ if } i \in \I_{\bullet}.
 \end{cases}
\end{equation}

\subsection{}

For $i\in \Iwhite$, we define 
\begin{align}  \label{eq:Zi}
Z_i &= \frac{1}{q_i^{-1} -q_i}  \  r_i \big(\T_{w_{\bullet}} (E_i) \big) \in \U_{\Iblack}^+. 
\end{align} 
It is known (cf. \cite{BK15}) that  $Z_i \neq 0$ (we thank Kolb for explaining this fact in detail). Balagovic--Kolb showed in  \cite[Proposition~2.5]{BK15} that 
\begin{equation}
  \label{eq:nu}
\sigma \tau (Z_i) =\nu_i Z_i \; \text{ with } \nu_i \in \{1, -1\}, \quad \forall i\in \I_\circ.
\end{equation}

\begin{conj}  \cite[Conjecture~2.7]{BK15}
  \label{conj:BK}
We have $\sigma \tau (Z_i) = Z_i$, that is, $\nu_i=1$, for all $i\in \I_\circ$.
\end{conj}
This conjecture will be established in full generality as Theorem~\ref{thm:nu=1} in Section~\ref{sec:lem}.
It is known \cite[Proposition 2.3]{BK15} that $\nu_i=1$ (that is, \cite[Conjecture~2.7]{BK15} holds) for $(\U, \Ui)$ of finite type. 

We shall assume Theorem~\ref{thm:nu=1} in the remainder of this section. 

\subsection{}

Throughout the paper, we make the following {\em basic assumptions}:
\begin{align}
  \label{eq:aij}
  |\langle i, j' \rangle |  \le 3,&  \; \forall i,j\in I,
  \quad
  \text{ or } \quad I_\bullet =\emptyset.
\end{align}

\begin{rem}
 \label{rem:assumption}
Condition~\eqref{eq:aij} is imposed so that explicit Serre type defining relations for $\Ui$ are available, and then the bar involution on $\Ui$ can be verified \cite{BK15} \cite{CLW18}. 
It is generally expected that the assumptions \eqref{eq:aij} can be removed eventually. 
\end{rem} 

The existence of the bar involution on $\Ui$ below was predicted in \cite{BW18a}.

\begin{lem} \cite{BK15} \cite{CLW18}
\label{lem:bar}
There is a unique anti-linear bar involution of the $\Q$-algebra $\Ui$, denoted by $\ov{\phantom{x}}$ or $\psi_{\imath}$, such that 
\begin{align*}
\psi_{\imath} (q) =q^{-1}, \quad
\psi_{\imath} (\ff_i) = \ff_i \; (i \in \I), \quad \ipsi(E_i) = E_i  \;(i \in \I_{\bullet}), \quad \ipsi (K_\mu) = K_{-\mu} \; (\mu \in Y^{\imath}).
\end{align*}
\end{lem}

We recall the following theorem (cf. \cite[Theorem 2.3]{BW18a}, \cite[Theorem~6.10]{BK19}, \cite[Theorem 4.8, Remark 4.9]{BW18b}).
\begin{thm}  
   \label{thm:Upsilon}
There exists a unique family of elements $\Upsilon_{\mu} \in \U^+_{\mu}$, 
such that $\Upsilon_0 =1$ and $\Upsilon = \sum_{\mu} \Upsilon_{\mu}$ satisfies the following identity (in $\widehat{\U}$):
\begin{equation}\label{eq:Upsilon}
 \ipsi(u) \Upsilon = \Upsilon \psi(u), \qquad \text{for all } u \in \Ui.
\end{equation}
Moreover, $\Upsilon_{\mu} =0$ unless ${\mu^\inv} = - \mu \in X$. 
\end{thm}
The formulation of the {\em quasi-$K$-matrix} $\Upsilon$ (called sometimes an {\em intertwiner}) was due to the authors \cite{BW18a}; its existence in full generality has been established in \cite{BK19} (also cf. \cite[Remark 4.9]{BW18b}).

\begin{rem}
Lemma~\ref{lem:bar} and Theorem~\ref{thm:Upsilon} were established in the literature under the assumption that $\nu_i=1$ for $i\in \I_\circ$; this assumption is now removed unconditionally thanks to Theorem~\ref{thm:nu=1}. Theorem~\ref{thm:Upsilon} for $\Ui$ with $\Iblack =\emptyset$ is new as its bar involution is only recently established in \cite{CLW18},  the usual proof in \cite{BW18a}-\cite{BK19} carries over. 
\end{rem}

\subsection{}
We define the modified version (i.e., idempotented version) $\Uidot$ of the $\imath$quantum groups following \cite[IV]{Lu94}. Some extra cares are required since the pairing between $X_\imath$ and $Y^\imath$ is not perfect in general, and we expand and correct slightly the definition given in \cite{BW18b}. We thank Hideya Watanabe for helpful remarks and suggestions.

We define an $X_\imath$-grading on $\Ui$ by assigning $B_i (i \in \I)$, $E_j (j \in \Iblack)$, $K_\mu (\mu \in Y^\imath)$ degree $-\overline{i'}$, $\overline{j'}$, $0$, respectively. One can easily check this is well-defined via Definition~\ref{def:Ui}. We denote by $\U^\imath(\zeta)$ the homogeneous subspace of $\Ui$ of degree $\zeta \in X_\imath$. Note that $\U^\imath(\zeta)$ is non-trivial if and only if $\zeta \in \overline{\Z[\I]} \subset X_\imath$.

For any $\lambda', \lambda'' \in X_\imath$ with $\zeta = \lambda '' -\lambda'$, we define 
\[
_{\lambda'}{\U}^\imath_{\lambda''} = \Ui(\zeta) / \big( \sum_{\mu \in Y^\imath} (K_\mu - q^{\langle \mu, \lambda' \rangle}) \Ui(\zeta) +  \sum_{\mu \in Y^\imath}  \Ui(\zeta) (K_\mu - q^{\langle \mu, \lambda'' \rangle}) \big).
\]
We denote by $\pi_{\lambda', \lambda''}: \Ui \rightarrow {}_{\lambda'}{\U}^\imath_{\lambda''}$ the natural quotient map, where $\pi_{\lambda', \lambda''}(x) =0$ if $x \not \in \Ui(\lambda'' - \lambda')$. We write $\one_{\lambda'} = \pi_{\lambda', \lambda'}(1)$ for the orthogonal idempotents. 

Following \cite[IV]{Lu94}, we then define an associative $\Qq$-algebra structure (without unit) on 
\[
	{\Uidot} = \bigoplus_{\lambda', \lambda'' \in X^\imath} {}_{\lambda'}{\U}^\imath_{\lambda''}.
\]
Moreover, the modified quantum group $\Udot$ is naturally a $(\Uidot, \Uidot)$-bimodule. For any $u \in \Uidot$ (or $\Ui)$ and $\one_{\lambda} \in \Udot$, we shall denote by $u \one_{\lambda} \in \Udot$ the action of the $u$ on $\one_\lambda$. The coproduct of $\Ui$ induces a similar structure on $\Uidot$ similar to \cite[23.1.5]{Lu94}.

The bar involution on $\Ui$ in Lemma~\ref{lem:bar} induces a bar involution $\ipsi$ on the $\Q$-algebra $\Uidot$ such that $\psi_{\imath} (q) =q^{-1}$ and
\begin{align*}
\psi_{\imath} (\ff_i \one_\zeta) = \ff_i \one_\zeta \; (i \in \I), \quad \ipsi(E_i \one_\zeta) 
= E_i \one_\zeta \;(i \in \I_{\bullet}), \quad \ipsi (\one_\zeta) = \one_{\zeta} \; (\zeta \in X_{\imath}).
\end{align*}

Recall $\Pa= \Pa_{\I_\bullet}$ and $\Padot$ from \S\ref{subsec:P}. Let $\U^+(w^\bullet)_{>}$ be the two-sided ideal of $\U^+$ generated by $E_i$ for $i \in \Iwhite$.
The composition map, denoted by $p_\imath= p_{\imath,\lambda}$, 
\begin{equation}
  \label{eq:comp3}
\Uidot \one_{\overline{\lambda}} \longrightarrow \Udot \one_{\lambda} 
 \longrightarrow  \Udot \one_\la\big/ \Udot   \U^+(w^\bullet)_{>} \one_\la \longrightarrow \Padot \one_{\lambda}, 
\end{equation}
is a $\Qq$-linear isomorphism; cf. \cite[\S3.8]{BW18b}.

\begin{definition}
  \label{def:mAUidot}
 We define $_{\mA} \Uidot$ to be the set of elements $u \in \Uidot$, such that $ u \cdot m \in {}_\mA\Udot$ for all $m \in {}_\mA\Udot$. Then $_{\mA} \Uidot$  is clearly a $\mA$-subalgebra of $\Uidot$ which contains all the idempotents $\one_\zeta$ $(\zeta \in X_\imath)$, and $_{\mA} \Uidot = \bigoplus_{\zeta \in X_\imath}\,  {}_{\mA} \Uidot \one_{\zeta}$.  
 \end{definition}


\subsection{Symmetries of $\Ui$ and $\Ui$-modules}
\subsubsection{}
  Recall the braid group operators $\T'_{i,e}$ and $\T''_{i,e}$, for $e=\pm 1$, from \cite{Lu94}. 
  
  \begin{prop}  \label{prop:rho}\cite[Theorem 4.2, Proposition~4.6, Proposition~4.13]{BW18b}
  Let $i\in \Iblack$, $j\in \Iwhite$, and $e =\pm 1$.
  \begin{enumerate}
  	\item The braid group operators $\T'_{i,e}$ and $\T''_{i,e}$ restrict to  isomorphisms of $\Ui$.
  	\item We have  $\T''_{i, e} (\Upsilon) = \Upsilon$ and $\T'_{i, e} (\Upsilon) = \Upsilon$.
  	\item Assume $\vs_{j} = q_j^{-1} \text{if } \kappa_j \neq 0$ and $\overline{\vs_j} = \vs_j^{-1}$ for the parameters. Then the anti-involution $\wp$ on $\U$ restricts to an anti-involution $\wp$ on $\U^\imath$.
  \end{enumerate}
  \end{prop}

 \subsubsection{} 
  We define the anti-linear involution $\sigma'_\imath$ of $\U$ as 
  \[
  \sigma'_\imath=\sigma \circ \tau \circ \psi : \U \longrightarrow \U.
  \]
  
  \begin{lem}
  We have 
$\sigma'_\imath (\vs_i  \T_{w_{\bullet}} (E_{\tau i}) \tK^{-1}_i )  =\vs_{\tau i}  \T_{w_{\bullet}} (E_{i}) \tK^{-1}_{\tau i} ,$ 
for $i\in \Iwhite$.
  \end{lem}
  
  \begin{proof}
  We have 
  \begin{align*}
  	\sigma'_\imath (\vs_i  \T''_{w_{\bullet},+1} (E_{\tau i}) \tK^{-1}_i ) &= \sigma \circ \tau (\overline{\vs_i} \,\,   \overline{\T''_{w_{\bullet},+1} (E_{\tau i})} \tK_i ) \\
	& = \sigma \circ \tau (\overline{\vs_i} \,\,   \T''_{w_{\bullet}, -1} (E_{\tau i}) \tK_i )\\
	& = \sigma (\overline{\vs_i}\,\,  \T''_{w_{\bullet}, -1} (E_{ i}) \tK_{\tau i} )\\
	& = \overline{\vs_i} \,\, \tK^{-1}_{\tau i}  \sigma \T''_{w_{\bullet}, -1} (E_{ i})\\
	&\stackrel{(a)}{=}  \overline{\vs_i}  q_i^{-\langle \tau i, w_{\bullet}   i' \rangle}  \T'_{w_{\bullet}, +1} (E_{i})  \tK^{-1}_{\tau i} \\
	&\stackrel{(b)}{=}   \overline{\vs_i}  q_i^{-\langle  i, w_{\bullet} \tau i' \rangle}   (-1)^{ \langle 2\rho^\vee_{\bullet},  i' \rangle } q_i^{-\langle i, 2\rho_{\bullet} \rangle} \T''_{w_{\bullet}, +1} (E_{ i})  \tK^{-1}_{\tau i} \\
	& \stackrel{(c)}{=}  \vs_{\tau i} \T''_{w_{\bullet}, +1} (E_{ i})  \tK^{-1}_{\tau i}.
  \end{align*}
  Here the equality $(a)$ follows \cite[\S37.2.4]{Lu94}, $(b)$ follows from  \cite[\S37.2.4]{Lu94} and a similar computation as \cite[Corollar~4.5]{BW18b}, and finally $(c)$ follows from \eqref{vs2}.
  \end{proof}
  
 \begin{prop}
 We have a $\Qq$-linear anti-involution $\sigma_\imath = \psi_\imath \circ \sigma_\imath'$ of the algebra $\Ui$, such that
$\sigma_\imath (F_i)  =F_{\tau i}$, $\sigma_\imath (E_i)=E_{\tau i} \; (i \in \Iblack)$,
$\sigma_\imath (B_i)=B_{\tau i} \; (i \in \Iwhite)$,
and $\sigma_\imath (K_\mu)=K_{- \tau \mu} \; (\mu \in Y^\imath)$.
 \end{prop}
 
 \begin{proof}
   Note that $\sigma_\imath' (F_i) =F_{\tau i}$, $\sigma'_\imath (E_i)=E_{\tau i}$, and $\sigma'_\imath (K_\mu)=K_{\tau \mu}$, for all $i \in \I, \mu \in Y$. It follows that 
   \[
   \sigma'_\imath (q)=q^{-1}, \; \sigma'_\imath (B_i)=B_{\tau i} \; (i\in \I_\circ), \quad
   \sigma'_\imath (K_\mu)=K_{\tau \mu}\; (\mu \in Y^\imath), 
   \]
   and hence $\sigma_\imath'$ restricts to an anti-linear anti-involution on the subalgebra $\Ui$. 
   
   The proposition follows immediately from the above and the definition of $\ipsi$ in Lemma~\ref{lem:bar}.
 \end{proof}
 Note $\sigma_\imath$ takes a particular neat form when $\tau=\id$, and it is strikingly similar to the anti-involution $\sigma$ on $\U$.

\subsubsection{}
Following \cite[\S4.5]{BW18b}, we consider the automorphism obtained by the composition 
$$
\vartheta =  \sigma \circ \wp \circ \tau : \U \longrightarrow \U, 
$$
which sends
\begin{equation}\label{eq:Ttauoemga}
\vartheta (E_i) = q_{\tau i}F_{\tau i} \tK_{- \tau i} , \quad  \vartheta (F_i) = q_{\tau i}  E_{\tau i} \tK_{\tau i},  \quad \vartheta (K_{\mu}) = K_{-\tau \mu}.
\end{equation}

For any  $\U$-module $M$, we define a new $\U$-module ${}^\vartheta M $ as follows: 
${}^\vartheta M$ has the same underlying $\Qq$-vector space as $M$ but we shall denote a vector in ${}^\vartheta M$ by ${}^\vartheta m$ for $m\in M$, and
the action of $u\in \U$ on ${}^\vartheta M$ is now given by $u \; {}^\vartheta m = {}^\vartheta (\vartheta^{-1} (u)m)$. 
Hence we have 
\begin{equation*}
\vartheta  (u) \; {}^\vartheta m = {}^\vartheta ( u m), \qquad \text{ for } u \in \U, m \in M.
\end{equation*}
As ${}^\vartheta M$ is simple if the $\U$-module $M$ is simple, one checks by definition that
\[
{}^\vartheta L(\lambda) \cong {}^{\omega}L(\lambda^{\tau}).
\]
\begin{rem}
Note that $\vartheta$ is {\em not} an automorphism of the subalgebra $\Ui$ in general (as we allow more general parameters $\vs_i$). Nevertheless $M$ and ${}^\vartheta M $ are both $\U$-modules and hence $\Ui$-modules by restriction.
\end{rem}

%

Let $g : X \longrightarrow  \Qq$ be such that \cite[(4.16)-(4.17)]{BW18b} hold. The function $g$ induces a $\Qq$-linear map from any weight $\U$-module $M = \oplus_{\mu \in X}M_\mu$ to itself: 
\begin{equation}
 \label{eq:g2}
\widetilde{g}: M\longrightarrow M, 
\qquad 
\widetilde{g}(m) = g (\mu) m, \quad \text{ for } m \in M_{\mu}.
\end{equation}

Recall we denote by $\eta_{\lambda}$ the highest weight vector in $L(\lambda)$.
 Let $\eta^{\bullet}_{\lambda}$ be the unique canonical basis element in $L(\lambda)$ of weight $\wb \lambda$. Recall $\lambda^\tau = \tau(\lambda)$. 
 Let $\mathcal C^{\text{hi}}$ be the BGG category of $\U$-modules with weights in $X$ and $\mathcal C'$ be the full subcategory of $\mathcal C^{\text{hi}}$ consisting of integrable $\U$-modules \cite[3.4.7, 3.5.1]{Lu94}. 

\begin{thm} (cf. \cite[Theorem~4.18]{BW18b}, \cite[Theorem 7.5]{BK19})
   \label{thm:mcT}
For any integrable $\U$-module $M = \oplus_{\mu \in X}M_\mu$ in $\mathcal C'$, we have the following isomorphism of $\Ui$-modules 
\[
\mc{T} := \Upsilon \circ \widetilde{g} \circ  \T^{-1}_{\wb} : M \longrightarrow {}^\vartheta M.
\]

In particular, we have the isomorphism of $\Ui$-modules
\[
\mc{T} : L(\lambda) \longrightarrow {}^\omega L(\lambda^{\tau}),
\qquad \etab_{\lambda} \mapsto \xi_{-\lambda^{\tau}}.
\]
\end{thm}
We note that the function ${g}$ above can be chosen such that $\mc{T}$ is an isomorphism of the $\mA$-form ${}_\mA L(\lambda) \longrightarrow {}^\omega_\mA L(\lambda^{\tau})$ once Corollary~\ref{cor:tensor} is established.

\begin{proof}
The same proof as \cite[Theorem~4.18]{BW18b} applies with minor modifications as specified below.

	The definition of the weight function $g$ in \cite[(4.15)]{BW18b} remains the same. However, we should replace $\vs_{\tau i}^{-1}$ by $\overline{\vs_{\tau i}}$ for the first identity in \cite[Lemma~4.16]{BW18b}, thanks to our relaxed conditions on parameters \eqref{parameters}. Otherwise, the proof of the lemma remains identical, and the original proof for \cite[Theorem~4.18]{BW18b} applies here verbatim.
\end{proof}


\section{Fundamental Lemma for QSP}
\label{sec:lem}

The goal of this section is to establish the Balagovic-Kolb conjecture \cite[Conjecture~2.7]{BK15} in full generality. The main tool here is Lusztig's theory of based modules and cells for $\U_{\Iblack}$ as developed in \cite[Chapters 27, 29]{Lu94}. 

Recall $Z_i$ from \eqref{eq:Zi}, and recall from \eqref{eq:nu} that $\sigma \tau (Z_i) =\nu_i Z_i$ for some $\nu_i \in \{1, -1\}$ and any $i\in \I_\circ.$
Balagovic-Kolb conjectured \cite[Conjecture~2.7]{BK15} (which is recalled in Conjecture~\ref{conj:BK}) that $\nu_i =1$ for all $i$ for $\U$ of Kac-Moody type. 
This conjecture looks rather technical and innocent but has been critical in several advances in the theory of QSP; for such reasons we have referred to the BK conjecture as the {\em Fundamental Lemma of QSP}.  

Several crucial results, such as the existence of bar involution with the parameters $\vs_i$ chosen to be in $q^\Z$ and the existence of quasi-K matrix, are established only under the assumption of this conjecture. Without the validity of the conjecture, it is unclear if suitable parameters for the QSP $(\U, \Ui)$ can be chosen to ensure the bar involution, quasi-K matrix and a meaningful specialization at $q=1$ to the usual symmetric pair. 
The conjecture was only known (based on results in \cite{BK15}) to hold for $\U$ of locally finite type, in the sense that all the real rank one Levi subalgebras of $\Ui$ are of finite type.  

\begin{thm} [Fundamental Lemma of QSP]
 \label{thm:nu=1}
For $\U$ of an arbitrary Kac-Moody type, we have $\nu_i=1$, that is, $\sigma \tau (Z_i) = Z_i$, for all $i\in \I_\circ$. 
\end{thm}

Let us prepare several lemmas. Denote 
\[
X_{\Iblack}^+ =\{\la \in X \mid \langle i, \la \rangle \in \N, \forall i \in \Iblack \}.
\]
Let  us fix $i\in \I_\circ$. We write $\alpha_i = i' \in X$ for notational consistency with \cite{BK15}. Then we may and shall regard $-\alpha_i \in X_{\Iblack}^+$ thanks to $i \cdot j \leq 0$ for $j\in \Iblack$. 
Recall $\xi_{\alpha_i}$ denotes the lowest weight vector in ${}^\omega L_{\Iblack}(-\alpha_i)= L_{\Iblack}(\wb \alpha_i)$. 

\begin{lem}
  \label{lem:nonzero}
The element $Z_i\in \U_{\Iblack}^+$ acts on $L_{\Iblack}(\wb \alpha_i)$ as a nonzero map.
In particular, we have 
$Z_i (\xi_{\alpha_i}) \neq 0$. (Since $Z_i\in \U^+_{\wb \alpha_i - \alpha_i}$ by \eqref{eq:Zi}, $Z_i (\xi_{\alpha_i})$ is of highest weight.)
\end{lem}

\begin{proof}
We owe this proof to Gail Letzter for her suggestions and references. The statement follows from a very special case of general results of Joseph and Letzter.

Let $\la \in X_{\Iblack}^+$. Denote by $K \in \U_{\Iblack}^0$ such that $KE_j K^{-1} =q_j^{- \langle j,\lambda \rangle} E_j$, for all $j \in \Iblack$. (This $K$ is relevant to \cite[Proposition~2.4]{BK15}.) Note the notation $\tau(\la)$ in Joseph-Letzter \cite{JL94, JL96} translates to $K^2$ here. It follows by  \cite[Corollary~3.3]{JL94} and \cite[(8.6)]{JL96}  that  
${\rm ad} (\U_{\Iblack}) (K^{2}) \longrightarrow {\rm End} \big(L_{\Iblack}(\la) \big)$ is injective.
(Actually this is an isomorphism for dimension reason.) 
 
 Setting $\la =\wb \alpha_i$, we have $K=\tK_i$, and $Z_i \tK_i^2 \in {\rm ad} (\U_{\Iblack}) (\tK_i^2)$ by \cite[(4.4)]{Ko14} (or \cite[(2.21)]{BK15}; by the injectivity above $Z_i \tK_i^2$ and hence $Z_i$ have nonzero images in ${\rm End} \big(L_{\Iblack}(\wb \alpha_i) \big)$. The lemma is proved. 
\end{proof} 

%
%

Recall from \cite[29.1.2]{Lu94} the two-sided ideals $\Udot_{\Iblack}[{}^> \wb \alpha_i]$ and $\Udot_{\Iblack}[{}^\geq \wb \alpha_i]$ of $\Udot_{\Iblack}$ are based submodule of $\Udot_{\Iblack}$. In addition, we have that $\Udot_{\Iblack}[{}^> \wb \alpha_i] \subset \Udot_{\Iblack}[{}^\geq \wb \alpha_i]$, and moreover $\Udot_{\Iblack}[{}^\geq \wb \alpha_i] \big/ \Udot_{\Iblack}[{}^> \wb \alpha_i]$ admits a canonical basis. The natural action of $ \Udot_{\Iblack}[{}^\geq \wb \alpha_i]$ on $L(\wb \alpha_1)$ factor through the projection 
\[
{\rm pr}: \Udot_{\Iblack}[{}^\geq \wb \alpha_i]  \longrightarrow \Udot_{\Iblack}[{}^\geq \wb \alpha_i] \big/ \Udot_{\Iblack}[{}^> \wb \alpha_i],
\]
 and it further induces an $\Qq$-algebra isomorphism below (cf. \cite[Proposition~ 29.2.2]{Lu94}):
\begin{equation}
  \label{eq:cell}
 \xymatrix{\Udot_{\Iblack}[{}^\geq \wb \alpha_i] \ar[d]^{{\rm pr}} \ar[drr] \\
\Udot_{\Iblack}[{}^\geq \wb \alpha_i] \big/ \Udot_{\Iblack}[{}^> \wb \alpha_i]  \ar[rr]^{\cong} && \text{End} \big(L(\wb \alpha_i) \big) }
\end{equation}

\begin{lem}
  \label{lem:cell}
We have 
\begin{enumerate}
\item
$Z_i \one_{\alpha_i} \in \Udot_{\Iblack} [\geq \wb \alpha_i]$;
\item
${\rm pr}(Z_i  \one_{\alpha_i}) \neq 0$ in $\Udot_{\Iblack}[{}^\geq \wb \alpha_i] \big/ \Udot_{\Iblack}[{}^> \wb \alpha_i]$;
\item
There exists a unique canonical basis element, denoted by $b_{\bullet}$, in the based module $\Udot_{\Iblack}[{}^\geq \wb \alpha_i] \big/ \Udot_{\Iblack}[{}^> \wb \alpha_i]$ of weight $\wb \alpha_i -\alpha_i$. Hence, ${\rm pr}(Z_i  \one_{\alpha_i}) = c \cdot b_{\bullet}$, where $c \in \Q(q)^\times$. 
\end{enumerate}
\end{lem}

\begin{proof}
Assume $Z_i \one_{\alpha_i}$ acts on $L(\lambda_2)$, for some $\la_2\in X_{\Iblack}^+$, by a nonzero map. Then we must have $\alpha_i \geq \wb \la_2$, or equivalently, $\la_2 \geq \wb \alpha_i$. Hence by \cite[Lemma~29.1.3]{Lu94}, we have $Z_i \one_{\alpha_i} \in \Udot_{\Iblack} [{}^\geq \wb \alpha_i]$, whence (1). 

By Lemma~\ref{lem:nonzero}, we have $(Z_i  \one_{\alpha_i} )(\xi_{\alpha_i}) \neq 0$, and it follows by \eqref{eq:cell} that ${\rm pr}(Z_i  \one_{\alpha_i}) \neq 0$, whence (2). 

Since the weight subspace of  $\text{End} \big(L(\wb \alpha_i) \big)$  of weight $\wb \alpha_i -\alpha_i$ is clearly one-dimensional, by the isomorphism in \eqref{eq:cell} we see that $\Udot_{\Iblack}[{}^\geq \wb \alpha_i] \big/ \Udot_{\Iblack}[{}^> \wb \alpha_i]$  has a one-dimensional weight subspace of weight $\wb \alpha_i -\alpha_i$, and (3) follows. 
\end{proof}

Now we are ready to prove the Fundamental Lemma of QSP. 

\begin{proof} [Proof of Theorem~\ref{thm:nu=1}]
The anti-involution $\sigma$ sends $\Udot_{\Iblack}[{}^\geq \la]$ onto $\Udot_{\Iblack}[{}^\geq -\wb \la]$, for any $\la \in X_{\Iblack}^+$, according to \cite[Lemma 29.3.1]{Lu94}. By the admissible condition (2) in \S\ref{subsec:adm}, the action of the diagram involution $\tau$ on $\I_{\bullet}$ coincides with the action of $-w_{\bullet}$, and thus $\tau$ induces an isomorphism $\Udot_{\Iblack}[{}^\geq -\wb\la] \rightarrow \Udot_{\Iblack}[{}^\geq \la]$, for any $\la \in X_{\Iblack}^+$. 

Hence, the composition $\sigma\tau$ preserves $\Udot_{\Iblack}[{}^\geq \la]$, for each $\la \in X_{\Iblack}^+$. It follows from \cite[Theorem~4.3.2]{Ka94} that $\sigma$ preserves the canonical basis of $\Udot_{\Iblack}$. The diagram automorphism $\tau$ preserve the canonical basis of $\Udot_{\Iblack}$ as well. Hence the canonical basis of $\Udot_{\Iblack}[{}^\geq \la]$ is stable under the action of $\sigma\tau$. 

 Therefore, by Lemma~\ref{lem:cell}(3), we must have $\sigma\tau(b_{\bullet})=b_{\bullet}$ since $\sigma\tau$ is weight-preserving. 

Hence, we have ${\rm pr} \big(\sigma\tau (Z_i \one_{\alpha_i}) \big) = \sigma\tau \big({\rm pr} (Z_i \one_{\alpha_i}) \big) =\sigma\tau  (c\cdot b_{\bullet}) =c\cdot b_{\bullet}$. 
On the other hand, by \eqref{eq:nu}, we have ${\rm pr} \big(\sigma\tau (Z_i \one_{\alpha_i}) \big) = {\rm pr} (\nu_i \cdot Z_i \one_{\alpha_i}) =\nu_i c\cdot b_{\bullet}$. By comparison we conclude that $\nu_i=1$.
\end{proof}

\section{The $\imath$divided powers}
 \label{sec:iDP}

\subsection{}
This section is devoted to a constructive proof of the existence of the so-called $\imath$divided powers. 

\begin{thm}\label{thm:iDP}
For any $i \in \I$ and $\zeta \in X_\imath$, there exists an element $B^{(n)}_{i, \zeta} \in {}_{\mA} \Uidot \one_{\zeta}$ satisfying the following 2 properties:
	\begin{enumerate}
		\item 	$\ipsi (B^{(n)}_{i, \zeta})=B^{(n)}_{i, \zeta}$;
		\item 	$ B^{(n)}_{i, \zeta} \one_{\lambda} = F^{(n)}_i \one_{\lambda} +\sum_{a < n}F^{(a)}_i {}_\mA \U^+\one_{\lambda} $, for $\one_{\lambda} \in {}_\mA\Uidot$ with $\overline{\lambda} =\zeta$.
	\end{enumerate}
\end{thm}

When $i \in \Iblack$, we can simply set $B^{(n)}_{i, \zeta}= F_i^{(n)} \one_{\zeta}$.   

We shall explicitly construct the elements $B^{(n)}_{i, \zeta} \in {}_{\mA} \Uidot \one_{\zeta}$ with the desired properties in Theorem~\ref{thm:iDP} for $i \in \Iwhite$ by separating the real rank one into 3 classes (this is a much rougher division than \cite{BW18a}, where 8 cases of $\imath$quantum groups of finite type of real rank one are enumerated). Two classes which are essentially known from  \cite{BW18a, BW18b} are treated in Subsection~\ref{subsec:rem}. Most of this section (\S\ref{subsec:boson}--\S\ref{subsec:Bin}) deals with the most nontrivial (potentially non-finite type) class when $\tau (i)=i\neq \wb i$ (Theorem~\ref{thm:iDP} in this case is summarized below as Theorem~\ref{thm:Bi}); the formula for this class is new even in finite type and allows a major simplification of \cite{BW18b}. 

\begin{rem}
The elements $B^{(n)}_{i, \zeta} \in {}_{\mA} \Uidot \one_{\zeta}$ in Theorem~\ref{thm:iDP} will be called {\em $\imath$divided powers}. The element $B^{(n)}_{i, \zeta}$ should be regarded as a leading term of the $\imath$canonical basis element $(1 \Diamond^\imath_{\zeta} F_i^{(n)})$ established in Theorem~\ref{thm:iCBUi}, since the $\imath$canonical basis is hard to compute and the equality $(1 \Diamond^\imath_{\zeta} F_i^{(n)}) = B^{(n)}_{i, \zeta}$ remains uncertain in general.
\end{rem}
We also refer to $E_i^{(n)} \one_{\zeta}$, for $i\in \Iblack$, as $\imath$divided powers.

\subsection{A $q$-boson algebra}
  \label{subsec:boson}

Let $i\in \Iwhite$ be such that $\tau (i)=i \neq \wb i$. Then we have $\T_{w_{\bullet}} (E_i)  \neq E_{i}$. 
This implies that $i\in I_{ns}$ in the notation of \cite{Ko14}, and it follows that $\kappa_i=0$ always.

\subsubsection{}

For such an $i \in \Iwhite$, recalling $Z_i$ from \eqref{eq:Zi}, for convenience below we define $\fZ_i = \vs_i Z_i$, that is, 
\begin{align}  \label{eq:fZi}
\begin{split}
\fZ_i & = \frac{\vs_i}{q_i^{-1} -q_i}  \  r_i \big(\T_{w_{\bullet}} (E_i) \big) \in \U_{\Iblack}^+.
\end{split}
\end{align} 

\begin{rem}
An element ${\CMcal Z}_i =- s(\tau(i))   r_i \big(\T_{w_{\bullet}} (E_i) \big)$ (in case $\tau=\id$) was introduced and studied in depth in \cite{Ko14} (see \cite[(3.10)]{BK15}). 
Our $\fZ_i$ is related to ${\CMcal Z}_i$ by 
\begin{equation}  \label{cZfZ}
\fZ_i = \frac1{q_i^{-1} -q_i} c_i {\CMcal Z}_i. 
\end{equation}
Indeed, for the theory of QSP throughout \cite{BK15} and this paper one only needs to use $Z_i$ instead of ${\CMcal Z}_i$. 
\end{rem}

\begin{lem} 
 \label{lem:FiEi}
Let $i \in \Iwhite$ be such that $\tau (i)=i \neq \wb i$.  We have
\begin{enumerate}
\item
$[F_i,  \vs_i \T_{w_{\bullet}} (E_i)] = \fZ_i \tK_i;$

\item
${}_ir \big(\T_{w_{\bullet}} (E_i) \big)=0$; 

\item
$\T_i \T_{w_{\bullet}} (E_i) \in \U^+$. 
\end{enumerate}
\end{lem}

\begin{proof}
Recall \cite[3.1.6]{Lu94}  that, for $x \in \U^+$, 
\[
 F_i x -x F_i = \frac{ r_i(x) \tK_i -\tK_{i}^{-1} \; {}_ir(x)}{q_i^{-1}-q_i}.
\]
The equivalence between (1) and (2)  follows from this. 
On the other hand, by  \cite[Proposition~38.1.6]{Lu94}, (2) and (3) are equivalent.

So it suffices to prove (3). The assumption $\T_{w_{\bullet}} (E_i)  \neq E_{ i}$ is equivalent to that $w_\bullet \alpha_i \neq \alpha_i$, 
an so we have $w_\bullet \alpha_i =\alpha_i + \sum_{j \in \Iblack} k_j \alpha_j\in \Phi_+ \backslash \{\alpha_i\}$,
where $\Phi_+$ denotes the set of positive roots of the Kac-Moody algebra $\g$.
Hence $s_i w_\bullet\alpha_i =w_\bullet \alpha_i -\langle i, w_\bullet \alpha_i \rangle \alpha_i \in \Phi_+$. 
This implies $\T_i \T_{w_\bullet} (E_i) =\T_{s_i w_\bullet} (E_i) \in \U^+$. 
\end{proof}


%
%
\subsubsection{}
  \label{subsec:boson2}
Set 
\[
Y_i =\vs_i \T_{w_{\bullet}} (E_i) \tK^{-1}_i.
\]

\begin{lem}
 \label{lem:q-H}
Let $i \in \Iwhite$ be such that $\tau (i)=i \neq \wb i$.  Then, 
\begin{enumerate}
\item
$\fZ_i$ commutes with $F_i$,  $Y_i$;  

\item
$
F_i   Y_i -q_i^{-2} Y_i F_i = \fZ_i.
$
\end{enumerate}
\end{lem}

\begin{proof}
(1) It follows from the presentation of $\Ui$ that  $[\fZ_i, B_i]=0$ (cf. \cite[({3.15})]{BK15} or  \cite[(7.7)]{Ko14}).
Hence it follows by \eqref{eq:Bi} that 
\[
[\fZ_i, F_i] +[\fZ_i, Y_i]  
 =[\fZ_i, B_i]=0.
\] 
As $[\fZ_i, F_i]$ and $[\fZ_i, Y_i]$ 
have distinct weights,
we must have $[\fZ_i, F_i] =[\fZ_i, Y_i] 
=0$. 

(2) follows since  $F_i Y_i -q_i^{-2} Y_i F_i =[F_i,  \vs_i \T_{w_{\bullet}} (E_i)] \tK^{-1}_i =\fZ_i$.
\end{proof}

We shall focus on two algebras $\He_i$ and $\TT_i$ below. 
\begin{enumerate}
\item[]
$\triangleright$ Denote by $\He_i$ the $\Qq$-subalgebra with 1 of $\U$  generated by $F_i, Y_i, \fZ_i$; we shall call $\He_i$ a {\em $q$-boson algebra}.  
\item[]
$\triangleright$ Denote by $\TT_i$ the $\Qq$-subalgebra with 1 of $\He_i$ generated by $B_i =F_i +Y_i$ and $\fZ_i$. 
\end{enumerate}
Clearly $\TT_i$ is also a $\Qq$-subalgebra of $\Ui$, and the bar map $\ipsi$ on $\Ui$ preserves the algebra $\TT_i$ thanks to \cite[Theorem~3.11]{BK15}. The algebras $\He_i$ and $\TT_i$ contains various integral elements of interest.

\subsection{Integral elements}

We continue to assume $i\in \Iwhite$  such that $\tau (i)=i \neq \wb i$.

\subsubsection{}
Denote $\fZ_i^{(n)} =\fZ_i^n/ [n]_i! $ and $Y_i^{(n)} = Y_i^n /  [n]_i! $, for $i \in \Iwhite$.  

\begin{prop}  \label{prop:Zint}
We have $Y_i^{(n)}  \in \aA \U$ and $\fZ_i^{(n)}  \in \aA \U_{\Iblack}$.
\end{prop}

\begin{proof}
Note $Y_i^{(n)} : =\vs_i^k  \frac{ (\T_{w_{\bullet}} (E_i) \tK_i^{-1})^n}{[n]_i!}  =\vs_i^k q^* \T_{w_{\bullet}} (E_i^{(n)} ) \tK_i^{-n} \in \aA \U$, where $*$ denotes a suitable integer.

It remains to show that $\fZ_i^{(n)}  \in \aA \U$,  since this implies that $\fZ_i^{(n)}  \in \U_{\Iblack} \cap \aA \U =\aA \U_{\Iblack}$.

Assume $xy -q_i^{-2}yx=z$ and $z$ commutes with both $x$ and $y$. Denote $x^{(n)} =x^n/ [n]_i!, y^{(n)} =y^n/ [n]_i!$ and $z^{(n)} =z^n/ [n]_i!$. 
Then one checks 
\begin{equation}
  \label{eq:xyz}
x^{(n)} y^{(m)} =
\sum_{a=0}^n q_i^{-2nm +(n+m)a  -\hf a(a-1)} y^{(m-a)} x^{(n-a)} z^{(a)}. 
\end{equation}
This is a variant of \cite[(3.1.2)]{Ka91}, which corresponds to our formula by specializing $z=1$ and $z^{(a)}=[a]_i^{-1}$ for all $a$. 
We rewrite the formula \eqref{eq:xyz} for $m=n$ as
\begin{equation}
  \label{eq:zn}
z^{(n)} = q_i^{\hf n(n-1)} 
\Big(
\sum_{a=0}^{n-1}  q_i^{-2n^2 +2na  -\hf a(a-1)} y^{(n-a)} x^{(n-a)} z^{(a)}
- x^{(n)} y^{(n)}
\Big). 
\end{equation}
By induction on $n$ and Equation \eqref{eq:zn}, we conclude that $z^{(n)} \in \aA \U$ if $x^{(k)}$ and $y^{(k)}$ lie in $\aA \U$, for all $k$. 

The above general formalism is applicable to $x=F_i$, $y=Y_i$, and $z=\fZ_i$, thanks to Lemma~\ref{lem:q-H}(2). 
Therefore, we conclude that $\fZ_i^{(n)} \in \aA \U$. 
\end{proof}

\begin{lem}
 \label{lem:ZA}
We have  $ \frac{Z_i}{q -q^{-1}}  \in \Ui \cap \aA \U$, and 
$ \frac{\fZ_i}{q -q^{-1}}  \in \Ui \cap \aA \U$.
\end{lem}

\begin{proof}
If suffices to prove the first statement as $\fZ_i =\vs_i Z_i$. 
By \cite[Proposition 3.5]{BK15}, (recalling $\CMcal Z_i =- s(\tau(i))   r_i \big(\T_{w_{\bullet}} (E_i) \big)$), 
we have 
\begin{equation}
  \label{eq:bZi}
 \psi(r_i (\T_{w_{\bullet}} (E_i) ) ) = \nu_i \ell_i r_i \big(\T_{w_{\bullet}} (E_i) \big)  = \ell_i r_i \big(\T_{w_{\bullet}} (E_i) \big),
 \end{equation} 
where $\nu_i=1$ thanks to Theorem~\ref{thm:nu=1}. By definition \eqref{eq:Zi}, 
$Z_i = \frac{1}{q_i^{-1} -q_i}  \  r_i \big(\T_{w_{\bullet}} (E_i) \big)$, and hence
\begin{equation}
  \label{eq:bZi2}
 \psi(Z_i) = - \ell_i  Z_i. 
 \end{equation} 
We recall $\ell_i =q^{(\alpha_i, \alpha_i -\wb \alpha_i -2 \rho_\bullet)} =q_i^{2 -\langle i, 2\rho_\bullet +\wb i' \rangle}$ is always an even power of $q$ by \cite[Remark 3.14]{BK15}. Set $m =(\alpha_i, \alpha_i -\wb \alpha_i -2 \rho_\bullet)/2 \in \Z$,and $Z'_i =q^m Z_i$. Then 
\[
\psi (Z'_i) =  -  Z'_i.
\]

As $Z'_i \in \aA \U^+$, we write $Z'_i =\sum_{k} d_k b_k$, a finite  $\A$-linear combination of canonical basis elements $b_k$, where $d_k \in \A$. It follows by $\psi (Z'_i) =  -  Z'_i$ that $\sum_{k} \ov{d_k} b_k = -\sum_{k} d_k b_k$. Therefore, $\ov{d_k} =- d_k$ for each $k$. Writing $d_k = \sum_{n\in \Z} a_{k;n} q^n$ with $a_{k;n}\in\Z$, we conclude that $a_{k;-n} =-a_{k;n}$ for all $n$. Hence $d_k = \sum_{n>0} a_{k;n} (q^n-q^{-n})$, and $\frac{d_k}{q -q^{-1}}  \in \A$. The lemma is proved. 
\end{proof}

 

%

\begin{example}
 \label{BII n=2}
Consider Satake diagram of type $BII$ with rank 2:
\[
   	\begin{tikzpicture}[baseline=0, scale=1.5]  
		\node at (-1.5,0) {$\circ$};
		\draw[-implies, double equal sign distance]  (-1.45, 0) to (-0.85, 0);
		\node at (-0.8,0) {$\bullet$};
		\node at (-1.5,.2) {\small 1};
		\node at (-0.8,.2) {\small 2};
	\end{tikzpicture}
\]
We have 
$
\imath( B_1) =F_1  +\vs_1  \T_{2} (E_1) \tK_1^{-1}.
$
Recall
$
\T_{2} (E_1)  =E_2^{(2)} E_1  -q^{-1}E_2 E_1 E_2 +q^{-2}  E_1 E_2^{(2)}.
$ 
A direct computation shows that 
$ 
[F_1, \T_{2} (E_1)] = (q^{-4}-q^{-2}) E_2^{(2)}  \tK_1,
$ 
Hence we have by Lemma~\ref{lem:FiEi} that  
\[
Z_1 =(q^{-4}-q^{-2}) E_2^{(2)},
\qquad
\frac{Z_1}{q-q^{-1}} =-q^{-3} E_2^{(2)} \in\aA \U_{\Iblack}.
\] 
\end{example}

\subsubsection{}
  
For $n\ge 0$, we define
\begin{align} 
  \label{eq:bi}
b_i^{(n)}  = \sum_{a=0}^n  q_i^{-a(n-a)}  Y_i^{(a)}F_i^{(n-a)} \in \aA \U. 
\end{align}

\begin{lem}
  \label{lem:bnAi} 
    We have $b_i^{(n)} \in \Ui \cap \aA \U$, for $i \in \Iwhite.$
\end{lem}

\begin{proof}
We know by definition \eqref{eq:bi} that $b_i^{(n)} \in \aA \U$. It remains to show that  $b_i^{(n)} \in \Ui$ by induction on $n$. The statement is trivial for $n=1$ as $b_i^{(1)}=B_i$.
It suffices to prove the following inductive formula for $n\ge 2$: \begin{align}
[n]_i b_i^{(n)} = b_i^{(n-1)} B_i - q_i^{2-n} \fZ_i b_i^{(n-2)}, 
\label{eq:bbb}
\end{align}
thanks to $B_i, \fZ_i \in \Ui$. 

It follows by Lemma~\ref{lem:q-H} and an induction on $m$ that 
\begin{equation}  \label{FmY}
F_i^{(m)}  Y_i -q_i^{-2m} Y_i F_i^{(m)} = q_i^{1-m} \fZ_iF_i^{(m-1)}.
\end{equation}

Recalling $B_i=F_i+Y_i$ and using \eqref{FmY} with $m=n-a-1$, we have
\begin{align*}
b_i^{(n-1)} B_i &= \sum_{a=0}^{n-1}  q_i^{-a(n-a-1)}  Y_i^{(a)}F_i^{(n-a-1)} B_i 
\\
&= \sum_{a=0}^{n-1}  q_i^{-a(n-a-1)} [n-a]_i Y_i^{(a)}F_i^{(n-a)}  
\\
&\qquad \quad+ \sum_{a=0}^{n-1}  q_i^{-a(n-a-1)}  Y_i^{(a)} 
\big(q_i^{2+2a-2n} Y_i F_i^{(n-a-1)} + q_i^{2+a-n} \fZ_iF_i^{(n-a-2)} \big)
\\
&\stackrel{(\star)}{=} \sum_{a=0}^n \big(q_i^{-a(n-a-1)}[n-a]_i + q_i^{-(a+1)(n-a)} [a]_i \big) Y_i^{(a)} F_i^{(n-a)} 
\\
&\qquad\quad + q_i^{2-n} \fZ_i \sum_{a=0}^{n-2}  q_i^{-a(n-a-2)}  Y_i^{(a)}F_i^{(n-a-2)}
\\
&= \sum_{a=0}^n [n]_i  q_i^{-a(n-a)} Y_i^{(a)} F_i^{(n-a)} 
 + q_i^{2-n} \fZ_i \sum_{a=0}^{n-2}  q_i^{-a(n-a-2)}  Y_i^{(a)}F_i^{(n-a-2)}
 \\
 &=q_i^{2-n} \fZ_i b_i^{(n-2)}  + [n]_i b_i^{(n)}.
\end{align*}
To obtain the equality $(\star)$ above, we have shifted the index of the (first half of) the second summand from $a$ to $a-1$. 
This proves the lemma. 
\end{proof}

\subsubsection{}

Recall the $\Q(q)$-subalgebra $\TT_i$ of $\Ui \cap \He_i$ from \S\ref{subsec:boson2}. 
It follows by induction on $n$ using \eqref{eq:bbb} that $b_i^{(n)} \in \TT_i$ for all $n$. 
Denote by $\TA_i$ the $\A$-subalgebra (with 1) of $\TT_i$ generated by $\fZ_i^{(n)}, b_i^{(n)}$, for all $n\ge 1$. 
It follows by \eqref{eq:bi} and Proposition~\ref{prop:Zint} that  $\TA_i \subset \aA \U$. 

\begin{prop}  \label{prop:bnbar}
We have   
\begin{align*}
\ipsi(b_i^{(n)}) 
&=   \sum_{k \ge 0}  q_i^{\frac{k(k+1)}{2}} \fZ_i^{(k)} b_i^{(n-2k)}.
\end{align*}
In particular, we have $\ipsi(b_i^{(n)})  \in \Ui \cap \TA_i\subset \Ui \cap \aA \U$.
%
\end{prop}

\begin{proof}
By Lemma~\ref{lem:bnAi}, the integrality statement for $\ipsi(b_i^{(n)})$ follows from the explicit formula in the lemma. 

To prove the formula, we proceed by induction on $n$. The case $n=1$ is clear as $b_i^{(1)} =B_i$ is $\ipsi$-invariant.

By \cite[Theorem~3.11(2)]{BK15}, we have $\ipsi (c_i {\CMcal Z}_i)  =q^{(i, \tau i)} c_{\tau i} {\CMcal Z}_{\tau i}$; 
this equality is transformed via \eqref{cZfZ} to be 
\begin{equation}
 \label{Zbar}
\ipsi (\fZ_i) =  \psi (\fZ_i) =  
- q_i^2 \fZ_i,
\end{equation}
thanks to $\fZ_i \in \U_{\Iblack}^+$.

Assume the statement holds for the cases of $\ipsi(b_i^{(k)})$ with $k\le n$, 
and we shall prove the formula for $\ipsi(b_i^{(n+1)})$.
By \eqref{eq:bbb}, we have 
\[
[n+1]_i b_i^{(n+1)}  =b_i^{(n)}B_i -q_i^{1-n} \fZ_i b_i^{(n-1)}.
\]
Applying the bar map $\ipsi$ to the above identity and using inductive assumptions on $\ipsi(b_i^{(n)})$ and $\ipsi(b_i^{(n-1)})$, we have

\begin{align*}
&[n+1]_i  \ipsi(b_i^{(n+1)})
\\
&\stackrel{\eqref{Zbar}}{=}   \ipsi(b_i^{(n)})B_i  +q_i^{n-1} q_i^{2} \fZ_i   \ipsi(b_i^{(n-1)})
\\
&=   \sum_{k \ge 0}    q_i^{\frac{k(k+1)}{2}} \fZ_i^{(k)} b_i^{(n-2k)} B_i
 +q_i^{n+1}  \sum_{k \ge 0}    q_i^{\frac{k(k+1)}{2}}  \fZ_i \fZ_i^{(k)} b_i^{(n-1-2k)}
\\
&\stackrel{(\star)}{=}   \sum_{k \ge 0}  q_i^{\frac{k(k+1)}{2}} 
 \Big ([n+1-2k]_i \fZ_i^{(k)}b_i^{(n+1-2k)} +q_i^{1+2k-n} [k+1]_i\fZ_i^{(k+1)} b_i^{(n-1-2k)} \Big)
\\&\qquad
 +q_i^{n+1}   \sum_{k \ge 0} q_i^{\frac{k(k+1)}{2}}  [k+1]_i \fZ_i^{(k+1)} b_i^{(n-1-2k)}
\\
&=   \sum_{k \ge 0}   q_i^{\frac{k(k+1)}{2}}  \Big ([n+1-2k]_i +q_i^{-k}
(q_i^{2k-1-n} +q_i^{n+1}) [k]_i \Big)  \fZ_i^{(k)} b_i^{(n+1-2k)} 
\\
&=[n+1]_i \sum_{k \ge 0} q_i^{\frac{k(k+1)}{2}} \fZ_i^{(k)} b_i^{(n+1-2k)}.
\end{align*}
In the identity ($\star$) above, we have used \eqref{eq:bbb} for $b_i^{(n-2k)} B_i$. 
The proof is completed. 
\end{proof}

\subsection{The $\imath$divided powers $B_i^{(n)}$}
  \label{subsec:Bin}

We continue to assume $i\in \Iwhite$ such that $\tau (i)=i \neq \wb i$.
We define
\begin{align}
B_i^{(n)} =\frac{q \ipsi( b_i^{(n)}) -q^{-1}  b_i^{(n)}}{q -q^{-1}}.
\label{eq:Bbn}
\end{align}

It follows from Proposition~\ref{prop:bnbar} that
\begin{align}
B_i^{(n)} = {b}_i^{(n)}  + \frac{q}{q -q^{-1}} \sum_{k \ge 1}  q_i^{\frac{k(k+1)}{2}} \fZ_i^{(k)} b_i^{(n-2k)}.
\label{eq:Bbn2}
\end{align}

\begin{lem}  \label{lem:fg}
Let $f, g \in \A$ be relatively prime.
Assume $u \in \U$ satisfies that $u/f, u/g \in \aA \U$. Then we have $u/(fg) \in \aA \U$. 
\end{lem}

\begin{proof}
Let ${\mathfrak B}$ be an $\A$-basis of $\aA \U$. Then $u =\sum_{b\in \mathfrak S} h_b b$, for $h_b \in \Qq$ and a finite subset $\mathfrak{S} \subset \mathfrak B$.
By assumption on $u/f$ and $u/g$, we have $h_b/f, h_b/g \in \A$, for $b\in \mathfrak S$. Since $f, g \in \A$ are relatively prime and $\A$ is 
a unique factorization domain, we have $h_b/(fg) \in \A$. Hence $u/(fg) =\sum_{b\in \mathfrak S} h_b/(fg) \cdot b \in \aA \U$. 
\end{proof}

\begin{lem}
  \label{lem:BinvZ}
For $n\ge 1$, we have $\ipsi(B_i^{(n)}) =B_i^{(n)}$, and moreover,
\[
B_i^{(n)} \in \Ui \cap \TA_i \subset \Ui \cap \aA \U.
\]
\end{lem}

\begin{proof}
It follows by definition~\eqref{eq:Bbn} that $\ipsi(B_i^{(n)}) =B_i^{(n)}$.

Clearly, $[n]_i!$ and $(q -q^{-1})$ are relatively prime in $\A$. By Lemma~\ref{lem:ZA}, we have $\frac{\fZ_i^n}{(q -q^{-1})}  \in \aA \U$. Since $\fZ_i^{(n)} =\frac{\fZ_i^n}{[n]_i!} \in \aA \U$ by Proposition~\ref{prop:Zint}, we conclude by Lemma~\ref{lem:fg} that $\frac{\fZ_i^{n}}{[n]_i! (q -q^{-1})}  \in \aA \U$, i.e., $\frac{\fZ_i^{(n)}}{q -q^{-1}}  \in \aA \U$.
Then it follows by \eqref{eq:Bbn2} and Lemma~\ref{lem:bnAi}  that   $B_i^{(n)} \in \TA_i \subset \aA \U$.
\end{proof}

Summarizing the discussions in \S\ref{subsec:boson}--\S\ref{subsec:Bin}, we have arrived at the following.

\begin{thm}\label{thm:Bi}
Assume $i\in \Iwhite$ such that $\tau (i)=i \neq \wb i$, and let $\zeta \in X_\imath$. Then $B_{i,\zeta}^{(n)} = B_i^{(n)} \one_{\zeta} \in {}_\mA \Uidot$ satisfies the conditions (1)-(2) in Theorem~\ref{thm:iDP}.
\end{thm}

\begin{proof}
Condition~(1)  in Theorem~\ref{thm:iDP} is satisfied by $B_{i,\zeta}^{(n)} = B_i^{(n)} \one_{\zeta} \in {}_\mA \Uidot$, thanks to Lemma~\ref{lem:BinvZ}. Condition~(2)  in Theorem~\ref{thm:iDP} follows from the formula \eqref{eq:Bbn2} for $B_i^{(n)}$ and the definition \eqref{eq:bi} of $b_i^{(n)}$ (note that $\fZ_i^{(k)} \in {}_\mA\U^+$ and $Y_i^{(a)} \one_{\mu} \in {}_\mA\U^+\one_\mu$).  
\end{proof}
%
%

%
%
\subsection{Additional $\imath$divided powers} \label{subsec:rem}

We consider the remaining 2 classes for $\imath$divided powers associated to $i\in \Iwhite$. 

\subsubsection{The class with $\tau(i) \neq i$.}

Let $i\in \Iwhite$ be such that $\tau (i)\neq i$. Then $\kappa_i=0$, and $B_i = F_i +\vs_i \T_{w_{\bullet}} (E_{\tau i})  \tK^{-1}_i $. We write again $Y_i : = \vs_i \T_{w_{\bullet}} (E_{\tau i})  \tK^{-1}_i$.
We have $F_i Y_i -q_i^{-2} Y_i F_i =[F_i,  Y_i] \tK^{-1}_i =0$.
Thanks to the $q$-binomial theorem, we define 
\[
B_{i,\zeta}^{(n)} = \frac{B_i^n}{[n]_i!} \one_{\zeta}
= \sum_{a=0}^n  q_i^{-a(n-a)}  Y_i^{(a)}F_i^{(n-a)} \one_\zeta \in {}_\mA \Uidot. 
\]

\subsubsection{The class with $\tau(i) = i =\wb i$.} 
Then $B_i = F_i +\vs_i E_i  \tK_i^{-1}+\ka_i \tK_i^{-1}$, for such $i\in \Iwhite$. 
This real rank one QSP is of local type AI. 
When $\vs_i =q_i^{-1}$, the existence of the $\imath$canonical basis (= $\imath$divided powers) in $\Uidot$ parametrized by $F_i^{(n)} \one_{\zeta}$, for $\zeta\in X_\imath$, was established in \cite{BW18a}.
With the more general parameter in the current setting, we can still obtain the precise inductive formula for the intertwiner $\Upsilon$ in the real rank one case as \cite[Lemma~4.6]{BW18a}. Afterwards, we can establish the $\imath$canonical basis of $\Uidot$ as in \cite{BW18a}, which will serve as the elements $B^{(n)}_{i, \zeta}$ desired in Theorem~\ref{thm:iDP}. Indeed, if we assume $\Upsilon = \sum_{k \ge 0} c_k E_i^{(k)}$, then we have 
\[
c_{k+1} = -q^{-k}(q-q^{-1}) (q^2 \vs_i [k] c_{k-1} + \kappa_i c_k), \quad \text{where } c_0 =1, c_{-1}=0.
\]
We clearly have $c_{k} \in \mA$ thanks to $\vs_{i}, \kappa_i \in \mA$.

The precise formulae for the $\imath$divided powers (= $\imath$canonical basis) in this real rank one case can be found in \cite{BeW18}  for distinguished parameters $\vs_i =q^{-1}_i$ and $\kappa_i$ a $q_i$-integer. The explicit formulas for the $n$-th $\imath$divided powers with a general parameter $\vs_i$ (and $\kappa_i=0$), for $n=2$, can be computed, though they remain to be computed in general for $n>2$.

\section{$\imath$Canonical bases for modules}
  \label{sec:iCBM}

In this section, we develop a theory of based $\Ui$-modules. To that end, we establish a key property that the quasi-$K$-matrix $\Upsilon$ preserves the integral forms of various based modules and their tensor products. 

\subsection{The $\mA$-forms}

\begin{definition}\label{def:AUi}
Let  $\aAp \Uidot$ be the $\A$-subalgebra of ${}_\mA\Uidot$ generated by the $\imath$divided powers $B^{(n)}_{i, \zeta} \,(i \in \I)$
and $E_j^{(n)} \one_\zeta \, (j \in \Iblack)$,  for all $n \ge 1$ and $\zeta \in X_\imath$.
\end{definition}

\begin{rem}
We shall see later that $\aAp \Uidot ={}_\mA\Uidot$ in Corollary~\ref{cor:AUi}. 
\end{rem}

Recall for $\lambda \in X$, we denote by $M(\lambda)$ the Verma modules of highest weight $\lambda$. We denote the highest weight vector by $\eta_\lambda$. The following is an $\imath$analogue of \cite[Lemma 2.2(1)]{BW16}.

\begin{lem}
\label{lem:surjA1}Let $(M, B(M))$ be a based $\U$-module. Let   $\la \in X$. Then, 
\begin{enumerate}
\item
for $b \in B(M)$, the $\Qq$-linear map 
$
\pi_b: \Uidot {\bf 1}_{\ov{|b|+\la}} \longrightarrow M \otimes M(\la),  \;
u \mapsto u(b \otimes \eta_\lambda),
$
restricts to an $\A$-linear map
$\pi_b: \aAp \Uidot {\bf 1}_{\ov{|b|+\la}} \longrightarrow \aA M \otimes_\A \aM(\la)$;

\item
we have $\sum_{b \in B(M)}  \pi_b (\aAp \Uidot {\bf 1}_{\ov{|b|+\la}})= \aA M \otimes_\A \aM(\la)$.
\end{enumerate}
\end{lem}

\begin{proof}
Recall $ \aAp \Uidot \subset  {}_\mA\Uidot$. Part (1) follows from Definition~\ref{def:mAUidot}.
We prove part (2) here following \cite[Lemma 2.2]{BW16}. We have  $\sum_{b \in B(M)}  \pi_b (\aAp \Uidot {\bf 1}_{\ov{|b|+\la}}) \subset \aA M \otimes_\A \aM(\la)$ by part (1).

Recall ${}_\mA\U^-$ has an increasing filtration 
$$
\A = {}_\mA\U^-_{\le 0} \subseteq {}_\mA\U^-_{\le 1} \subseteq \cdots \subseteq {}_\mA\U^-_{\le N} \subseteq \cdots
$$
where ${}_\mA\U^-_{\le N}$ denotes the $\A$-span of $\{F_{i_1}^{(a_1)}\ldots F_{i_n}^{(a_n)}   
| a_1+\ldots + a_n\le N, i_1, \ldots, i_n \in I\}.$
This induces an  increasing filtration $\{ \aM(\la)_{\le N} \}$ on $\aM(\la)$. We shall prove by induction on $N$ that ${}_\mA M \otimes_\mA \aM(\la)_{\le N} \subset \sum_{b \in B(M)}  \pi_b (\aAp \Uidot {\bf 1}_{\ov{|b|+\la}})$.

Let $b \otimes \big(F_{i_1}^{(a_1)}\ldots F_{i_n}^{(a_n)} \eta_{\lambda} \big) \in {}_\mA M \otimes_\mA \aM(\la)_{\le N}$. Recall the element $B^{(a_i)}_{i_1, \zeta}$ in Theorem~\ref{thm:iDP} with appropriate $\zeta \in X^\imath$. Thanks to \cite[Lemma 2.2]{BW16} and Theorem~\ref{thm:iDP}, we have 
\[
B^{(a_i)}_{i_1, \zeta} \Big( b \otimes \big(F_{i_2}^{(a_2)}\ldots F_{i_n}^{(a_n)} \eta_{\lambda} \big) \Big) \in b \otimes \big(F_{i_1}^{(a_1)}\ldots F_{i_n}^{(a_n)} \eta_{\lambda} \big)  +  {}_\mA M \otimes_\mA \aM(\la)_{\le N-1}.
\]
The lemma follows.
\end{proof}

For $\la \in X^+$, we abuse the notation and denote also by $\eta_\la$ the image of $\eta_\lambda$ under the projection $p_\la: M(\la) \rightarrow L(\la)$. 
Note that $p_\la$ restricts to $p_\la: \aM(\la) \rightarrow \aL(\la)$. The next corollary follows from Lemma~\ref{lem:surjA1}. 

\begin{cor}
\label{cor:surjA2}
Let $\la \in X^+$, and let $(M, B(M))$ be a based $\U$-module. Then, 
\begin{enumerate}
\item
for $b \in B(M)$, the $\Qq$-linear map 
$
\pi_b: \Uidot {\bf 1}_{\ov{|b|+\la}} \longrightarrow M \otimes L(\la),  
\; u \mapsto u(b \otimes \eta_\la)$, restricts to an $\A$-linear map
$\pi_b: \aAp \Uidot {\bf 1}_{\ov{|b|+\la}} \longrightarrow \aA M \otimes_\A \aL(\la)$; 

\item
we have $\sum_{b \in B(M)}  \pi_b (\aAp \Uidot {\bf 1}_{\ov{|b|+\la}})= \aA M \otimes_\A \aL(\la)$.
\end{enumerate}
\end{cor}

\subsection{Integrality of actions of $\Upsilon$}

\subsubsection{}
The quasi-$K$-matrix $\Upsilon \in \widehat{\U}^+$ induces a well-defined $\Qq$-linear map on $M \otimes L(\la)$: 
\begin{equation}
   \label{eq:K-ML}
\Upsilon: M \otimes L(\la) \longrightarrow M \otimes L(\la),
\end{equation}
for any $\la \in X^+$ and any weight $\U$-module $M$ whose weights are bounded above. 

Recall \cite[\S5.1]{BW18b} that a $\Ui$-module $M$ equipped with an anti-linear involution $\ipsi$ 
is called {\em involutive} (or {\em $\imath$-involutive}) if
$$
\ipsi(u m) = \ipsi(u) \ipsi(m),\quad \forall u \in \Ui, m \in M.
$$

\begin{prop}
  \label{prop:compatibleBbar}
Let $(M,B)$ be a based $\U$-module whose weights are bounded above. We denote the bar involution on $M$ by $\psi$. Then $M$ is an $\imath$-involutive  
$\Ui$-module with involution
\begin{equation}
\label{ibar}
\ipsi := \Upsilon \circ \psi.
\end{equation}
\end{prop}
\begin{proof}
Note that since the weights of $M$ are bounded above, the action $\Upsilon : M \rightarrow M$ is well defined. The rest of the argument follows from \cite[Proposition~5.1]{BW18b}.
\end{proof}

\subsubsection{} 

\begin{prop}
\label{prop:quasi-KZ}
Let $(M,B)$ be a based $\U$-module whose weights are bounded above such that $\Upsilon$ preserves the $\A$-submodule $\aM$. Then the  $\Qq$-linear map $\ipsi := \Upsilon \circ \psi$ preserves the $\A$-submodule $\aM \otimes_\A \aL(\la)$, for any $\la\in X^+$. 
\end{prop}

\begin{proof}
The $\U$-module $M\otimes L(\la)$ is involutive with the involution $\psi:=\Theta \circ (\ov{\phantom{x}} \otimes \ov{\phantom{x}})$; see \cite[27.3.1]{Lu94}. It follows by \cite[Proposition~2.4]{BW16} that $\psi$ preserves the $\A$-submodule $\aM \otimes_\A \aL(\la)$.

Regarded as $\Ui$-module $M\otimes L(\la)$ is $\imath$-involutive with the involution $\ipsi:=\Upsilon \circ \psi$. 
We now prove that $\ipsi$ preserves the $\A$-submodule $\aM \otimes_\A \aL(\la)$.

By  Corollary~\ref{cor:surjA2}(2), for any $x \in \aM \otimes_\A \aL(\la)$, we can write $x=\sum_k u_k (b_k\otimes \eta_\la)$, for $u_k \in \aAp \Uidot$ and $b_k \in B$. 
Since $M \otimes L(\la)$ is $\imath$-involutive, we have
\begin{align}
  \label{eq:psipsi}
\ipsi (x) =\sum_k \ipsi (u_k) \ipsi (b_k \otimes \eta_\la) 
=\sum_k \ipsi (u_k) \Upsilon \psi (b_k \otimes \eta_\la)
=\sum_k \ipsi (u_k) (\Upsilon  b_k \otimes \eta_\la),
\end{align}
where we have used $\Delta (\Upsilon) \in \Upsilon \otimes 1 + \U \otimes \U^+_{>0}$ by \cite[3.1.4]{Lu94}. 
By assumption we have $\Upsilon  b_k \in \aM$ and it follows by definition of $\aAp \Uidot$ that $\ipsi (u_k) \in  {}_\mA\Uidot$. Applying Corollary~\ref{cor:surjA2}(1) again to \eqref{eq:psipsi}, we obtain that $\ipsi (x) \in \aM \otimes_\A \aL(\la)$. The proposition follows. 
\end{proof}

\begin{cor}
 \label{cor:Zform}
Let $(M,B)$ be a based $\U$-module whose weights are bounded above such that $\Upsilon$ preserves the $\A$-submodule $\aM$. Then $\Upsilon$ preserves the $\A$-submodule $\aM \otimes_\A \aL(\la)$. In particular, $\Upsilon$ preserves the $\A$-submodule $\aL(\la)$ of $L(\la)$.
\end{cor}
\begin{proof}
Recall $\Upsilon =\ipsi \circ \psi$. The corollary follows from Proposition~\ref{prop:quasi-KZ} and the fact that $\psi$ preserves the $\A$-submodule $\aM \otimes_\A \aL(\la)$.
\end{proof}

\begin{cor} 
\label{cor:tensor}
Let $\lambda_i \in X^+$ for $1 \le i \le \ell$.  The involution $\ipsi$ on the $\imath$-involutive $\Ui$-module $L(\la_{1})\otimes \ldots \otimes L(\la_\ell)$ preserves the $\mA$-submodule ${}_\mA L(\la_{1})\otimes_{\mA} \ldots \otimes_{\mA} {}_\mA L(\la_\ell)$.
\end{cor}

\begin{proof}
Thanks \cite[Theorem~2.9]{BW16} we know that $L(\la_{1})\otimes \ldots \otimes L(\la_\ell)$ is a based $\U$-module whose weights are bounded above. Then the corollary follows by applying Proposition~\ref{prop:quasi-KZ} consecutively. 
\end{proof}

\begin{thm}\label{thm:int}
Assume $(\U, \Ui)$ is of finite type. Write $\Upsilon =\sum_{\mu \in \Z\Pi} \Upsilon_\mu$. Then we have $\Upsilon_\mu \in \aA \U^+$ for each $\mu$. 
\end{thm}

\begin{proof}
Follows by Corollary~\ref{cor:Zform} and applying $\Upsilon$ to the lowest weight vector $\xi_{-w_0\la} \in {}_\mA L(\la)$, for $\la\gg 0$ (i.e., $\la \in X^+$ such that $\langle i, \la \rangle \gg 0$ for each $i$). 
\end{proof}

\begin{rem} 
\label{rem:int}
 Theorem~\ref{thm:int}, which allows general parameters $\vs_i \in \Z[q,q^{-1}]$ as in \eqref{parameters}, generalizes \cite[Theorem~D]{BW18b}, which required $\vs_i \in \pm q^\Z$. Our current approach (which is based on the new $\imath$divided powers developed in Section~\ref{sec:iDP}) avoids the tedious case-by-case verification in the 8 cases of real rank one QSP in \cite[Appendix~A]{BW18b}.
\end{rem}

\subsection{$\imath$Canonical bases on modules}

\subsubsection{}
Let us first recall the following definition of based $\Ui$-modules in \cite[Definition~1]{BWW20}. Let $\mathbf{A}= \Q[[q^{-1}]] \cap \Qq$.   

 We call a $\Ui$-module $M$ a weight $\Ui$-module, if $M$ admits a direct sum decomposition $M = \oplus_{\lambda \in X_\imath} M_{\lambda}$ such that, for any $\mu \in Y^\imath$, $\lambda \in X_\imath$, $m \in M_{\lambda}$, we have $K_{\mu} m = q^{\langle \mu, \lambda \rangle} m$. 

\begin{definition}\label{ad:def:1}
Let $M$ be a weight $\Ui$-module over $\Qq$ with a given $\Qq$-basis $\B^\imath$. The pair $(M, \B^\imath)$ is called a based $\Ui$-module if the following conditions are satisified:
	\begin{enumerate}
		\item 	$\B^\imath \cap M_{\nu}$ is a basis of $M_{\nu}$, for  any $\nu \in X_\imath$;
		\item 	The $\mA$-submodule ${}_{\mA}M$ generated by $\B^\imath$ is stable under ${}_{\mA}\Uidot$;
		\item 	$M$ is $\imath$-involutive; that is, the $\Q$-linear involution $\ipsi: M \rightarrow M$ defined by 
		$\ipsi(q)= q^{-1}, \ipsi(b) = b$ for all $b \in \B^\imath$ 
		is compatible with the $\Uidot$-action, i.e., $\ipsi(um) = \ipsi(u) \ipsi(m)$, for all $u\in \Uidot, m\in M$;
		\item 	Let $L(M)$ be the $\mathbf{A}$-submodule of $M$ generated by $\B^\imath$. Then the image of $\B^\imath$ in $L(M)/ q^{-1}L(M)$ forms a $\Q$-basis in $L(M)/ q^{-1}L(M)$. 
	\end{enumerate}
\end{definition}	
We shall denote by $\mathcal L(M)$ the $\Z[q^{-1}]$-span of $\B^\imath$; then $\B^\imath$ forms a $\Z[q^{-1}]$-basis for $\mathcal L(M)$. We also have the obvious notions of based $\Ui$-submodules and based quotient $\Ui$-modules.

\subsubsection{}
We have the following generalization of \cite[Theorem~5.7]{BW18b} to QSP of Kac-Moody type.

\begin{thm}  
\label{thm:iCBmodule}
Let $(M,B)$ be a based $\U$-module whose weights are bounded above. Assume the involution $\ipsi$ of $M$ from Proposition~\ref{prop:compatibleBbar}  preserves the $\A$-submodule $\aM$. 

\begin{enumerate}
\item
The $\Ui$-module $M$ admits a unique basis (called $\imath$canonical basis)
$
B^\imath  := \{b^\imath \mid b \in B \}
$
which is $\ipsi$-invariant and of the form
\begin{equation} \label{iCB}
b^\imath = b +\sum_{b' \in B, b <  b'}
t_{b;b'} b',
\quad \text{ for }\;  t_{b;b'} \in q^{-1}\Z[q^{-1}].
\end{equation}

\item
$B^\imath$ forms an $\mA$-basis for the $\mA$-lattice ${}_\mA M$ (generated by $B$), and
forms a $\Z[q^{-1}]$-basis for the $\Z[q^{-1}]$-lattice $\mc{M}$ (generated by $B$).
\item		$(M, B^\imath)$ is a based  $\Ui$-module, where we call $B^\imath$ the $\imath$canonical basis of $M$.
\end{enumerate}
\end{thm}  

\subsubsection{}
 Recall from Theorem~\ref{thm:based-Li} the based $\U$-submodule $L(w\lambda, \mu)$, for $\lambda, \mu \in X^+$, and $w\in W$.
 
 \begin{thm} \label{thm:iCBbased}
  Let $\lambda, \mu, \lambda_i \in X^+$ for $1 \le i \le \ell$, and $w\in W$. 
 	\begin{enumerate}	
		\item 	$L(\lambda_1) \otimes \ldots \otimes L(\la_\ell)$ is a based $\Ui$-module, with the  $\imath$canonical basis defined as Theorem~\ref{thm:iCBmodule}. 
		\item 	$L(w\lambda, \mu)$ is a based $\Ui$-submodule of $L(\lambda) \otimes L(\mu)$.
	\end{enumerate}
 \end{thm}
 
\begin{proof}
It suffices to verify the assumptions of Theorem~\ref{thm:iCBmodule}. Indeed, it is clear that both $L(\lambda_1) \otimes \ldots \otimes L(\la_\ell)$ and $L(w\lambda, \mu)$ have weights bounded above. It follows from Corollary~\ref{cor:tensor} that $\ipsi$ preserves the $\mA$-submodule ${}_\mA L(\la_1) \otimes_\mA \ldots \otimes_\mA {}_\mA L(\la_\ell)$, and hence preserves ${}_\mA L(w\lambda, \mu)$ too. Therefore both $L(\lambda_1) \otimes \ldots \otimes L(\la_\ell)$ and $L(w\lambda, \mu)$ are based $\Ui$-modules.  It is obvious that $L(w\lambda, \mu)$ is a based $\Ui$-submodule of $L(\lambda) \otimes L(\mu)$. The theorem is proved. 
\end{proof}

For $\lambda, \mu \in X^+$, we shall denote by 
\[
L^\imath(\la,\mu) =L(w_\bullet \la, \mu), 
\]
which is a based $\U$-module and also a based $\Ui$-module thanks to Theorem~\ref{thm:based-Li}.

\subsection{The element $\Theta^\imath$}
%

We define 
\begin{equation}\label{ad:eq:1}
\Theta^\imath = \Delta (\Upsilon) \cdot \Theta \cdot (\Upsilon^{-1} \otimes \id).
\end{equation}
Note that $\Theta^\imath$ only lies in a completion of $\U \otimes \U$, whose definition can be found in \cite[Section~3.1]{Ko20}. Even though only finite types were considered in \cite{Ko20}, the generalization to Kac-Moody is straightforward.

We can write 
\begin{equation}
  \label{eq:Thetamu}
\Theta^\imath = \sum_{\mu \in \N\I}\Theta^\imath_{\mu}, \qquad
\text{ where } \Theta^\imath_{\mu} \in  \U \otimes \U^+_\mu.
\end{equation}
The following result first appeared in \cite[Proposition~3.5]{BW18a} for the quantum symmetric pairs of (quasi-split) type AIII/AIV.
\begin{lem}\cite[Proposition~3.10]{Ko20}\label{ad:lem:1}
We have $\Theta^\imath_{\mu} \in \Ui \otimes \U^+_\mu$, for all $\mu$. In particular,  we have $\Theta^\imath_{0} = 1\otimes 1$. 
\end{lem}
The proof of \cite[Proposition~3.10]{Ko20} remains valid in the Kac-Moody setting. (We thank Kolb for the confirmation.)

\begin{thm}\label{thm:Uidotco}
Let $M$ be a based $\Ui$-module, and $\lambda \in X^+$. Then $\ipsi \stackrel{\rm def}{=} \Theta^\imath \circ (\ipsi \otimes \psi)$ is an anti-linear involution on $M \otimes L(\lambda)$, and $M \otimes L(\lambda)$ is  a based $\Ui$-module with a bar involution $\ipsi$.
\end{thm}

\begin{proof}
The anti-linear operator $\ipsi = \Theta^\imath \circ (\ipsi \otimes \psi): M \otimes L(\lambda) \rightarrow M \otimes L(\lambda)$ is well defined thanks to Lemma~\ref{ad:lem:1} and the fact that the weights of $L(\lambda)$ are bounded above. Then entirely similar to \cite[Proposition~3.13]{BW18a}, we see that $\ipsi^2=1$ and $M \otimes L(\lambda)$ is $\imath$-involutive in the sense of Definition~\ref{ad:def:1}(3). 

We now prove that $\ipsi$ preserves the $\mA$-submodule ${}_\mA M \otimes_{\mA} {}_\mA L(\lambda)$. By assumption, $(M, \B^\imath(M))$ is a based $\Ui$-module. For any $b \in \B^\imath(M)$, we define 
\[
\pi_b: {}_\mA \Uidot \rightarrow {}_\mA M \otimes_{\mA} {}_\mA L(\lambda), \quad u \mapsto u (b \otimes \eta_{\la}).
\]
Indeed, $\pi_b$ is well defined, since $u \in {}_{\mA} \Uidot$ if and only if $u \cdot \one_{\mu} \in {}_{\mA} \Udot$ for each $\mu \in X$; cf. \cite[Lemma~ 3.20]{BW18b} (see Definition~\ref{def:mAUidot} for ${}_{\mA} \Uidot$). 


Note that $\ipsi (b \otimes \eta_{\lambda}) =  (b \otimes \eta_{\lambda})$ for any $b \in \B^\imath(M)$. Following the proof of Lemma~\ref{lem:surjA1}, we have $\sum_{b \in \B^\imath(M)} \pi_b( {}_\mA '\Uidot ) =  {}_\mA M \otimes_{\mA} {}_\mA L(\lambda)$. Hence we also have $\sum_{b \in \B^\imath(M)} \pi_b( {}_\mA \Uidot ) =  {}_\mA M \otimes_{\mA} {}_\mA L(\lambda)$, since ${}_\mA '\Uidot \subset {}_\mA \Uidot$. Then the same strategy of Proposition~\ref{prop:quasi-KZ} implies that $\ipsi$ preserves the $\mA$-submodule ${}_\mA M \otimes_{\mA} {}_\mA L(\lambda)$.

We write  $\B = \{b^- \eta_{\la} \vert b \in \B(\lambda)\}$ for the canonical basis of $L(\lambda)$. Following the same argument as for \cite[Theorem~4]{BWW20}, we conclude that: 
\begin{enumerate}
\item
for $b_1 \in \B^\imath, b_2\in \B$, there exists a unique element $b_1\diamondsuit_\imath b_2$ which is $\ipsi$-invariant such that $b_1\diamondsuit_\imath b_2\in b_1\otimes b_2 +q^{-1}\Z[q^{-1}] \B^\imath \otimes \B$;
\item
we have $b_1\diamondsuit_\imath b_2 \in b_1\otimes b_2 +\sum\limits_{(b'_1,b'_2) \in \B^\imath \times \B, |b_2'| < |b_2|} q^{-1}\Z[q^{-1}] \, b_1' \otimes b_2'$;
\item
$\B^\imath \diamondsuit_\imath \B :=\{b_1\diamondsuit_\imath b_2 \mid b_1 \in \B^\imath, b_2\in \B \}$ forms a $\Qq$-basis for  $M \otimes L(\lambda)$, an $\mA$-basis for ${}_\mA M  \otimes_{\mA} {}_\mA L(\lambda)$, and a $\Z[q^{-1}]$-basis for $\mathcal L(M)  \otimes_{\Z[q^{-1}]} \mathcal L(\lambda)$; 
\item
$(M\otimes L(\lambda), \B^\imath \diamondsuit_\imath \B)$ is a based $\Ui$-module.
\end{enumerate}
\end{proof}

\section{Canonical bases on the modified $\imath$quantum groups}
  \label{sec:iCBU}

In this section, we formulate the main definition and theorems on canonical bases on the modified $\imath$quantum groups. The formulations are straightforward generalizations of the finite type counterparts in \cite[Section 6]{BW18b} (with mild modifications), and the reader is encouraged to be familiar with  \cite[Section 6]{BW18b} first. Thanks to the new results established in the previous sections, they are now valid in the setting of QSP of Kac-Moody type. 

\subsection{The modified $\imath$quantum groups}

 Recall the partial order $\leq$ on the weight lattice $X$ in \eqref{eq:leq}. The following proposition generalizes \cite[Propositions~6.8, 6.12, 6.13, 6.16]{BW18b} to Kac-Moody types.
 
 \begin{prop}\label{prop:props}
 Let $\lambda, \mu \in X^+$ and let $\zeta = \wb \lambda +\mu$ and $\zeta_\imath = \overline{\zeta}$. 
 	\begin{enumerate}
 		\item	
 		The $\imath$canonical basis of $L^{\imath}(\lambda, \mu)$ is the basis $\B^\imath(\la, \mu) =\big\{ (b_1 \diamondsuit_{\zeta_\imath} b_2 )_{w_\bullet\la,\mu}^{\imath} \vert 
	(b_1, b_2) \in \B_{\Iblack} \times \B \big\} \backslash \{0\}$, where 
 $(b_1 \diamondsuit_{\zeta_\imath} b_2 )_{w_\bullet\la,\mu}^{\imath}$ is $\ipsi$-invariant and lies in 
	\begin{equation*} 
	 (b_1 \diamondsuit_{\zeta} b_2 ) (\eta_\lambda^\bullet \otimes \eta_{\mu})    +\!\!\! \sum_{|b_1'|+|b_2'| \le |b_1|+|b_2|} \!\!\!q^{-1}\Z[q^{-1}] (b'_1 \diamondsuit_{\zeta} b'_2 ) (\eta_\lambda^\bullet \otimes \eta_{\mu}). 
	\end{equation*}
 	\item We have the projective system $\big \{ L^\imath(\lambda+\nu^\tau, \mu+\nu)  \big \}_{\nu \in X^+}$ of $\Ui$-modules, where 
\begin{equation*} 
\pi_{\nu+ \nu_1, \nu_1}: L^{\imath} (\lambda+\nu^\tau+\nu_1^\tau, \mu+\nu+\nu_1) \longrightarrow L^{\imath} (\lambda+\nu^\tau, \mu+\nu), \quad \nu, \nu_1 \in X^+,
\end{equation*}
is the unique homomorphism of $\Ui$-modules such that 
\[\pi(\etab_{\lambda+\nu^\tau+\nu_1^\tau} \otimes \eta_{\mu+\nu+\nu_1}) = \etab_{\lambda+\nu^\tau} \otimes \eta_{\mu+\nu}.\] 

	\item The projective system in (2) is asymptotically based in the following sense: for fixed $(b_1, b_2) \in \B_{\Iblack} \times \B$ and any $\nu_1 \in X^+$, as long as $\nu \gg 0$, we have 
	\[ 
\pi_{\nu+ \nu_1, \nu_1} \big((b_1 \diamondsuit_{\zeta_\imath} b_2 )_{w_\bullet(\lambda+\nu^\tau+\nu_1^\tau),\mu+\nu+\nu_1}^{\imath}\big) =  \big((b_1 \diamondsuit_{\zeta_\imath} b_2 )_{w_\bullet(\lambda+\nu^\tau),\mu+\nu}^{\imath}\big).
	\]
 \end{enumerate}
 \end{prop}
 
 \begin{proof}
 Claim (1) is just a reformulation of Theorem~\ref{thm:iCBbased}(2). Claim (2) follows by the same proof as \cite[Proposition~6.12]{BW18b}. 
 
 Claim (3) is essentially the same as \cite[Proposition~6.16]{BW18b} with a mild modification which we now explain. 
 We used the finite-dimensionality of the module $L(\nu^\tau, \nu)$ four lines below \cite[(6.5)]{BW18b}. But this can be replaced by the fact that only finitely many $a(b', b'')$ are non-zero {\em loc cit}. The rest is exactly the same. 
Note that the (fixed) parameters plays no essential role here as we are taking  $\nu \gg 0$.
 \end{proof}

\begin{thm}  \cite[Theorem 6.17]{BW18b}
  \label{thm:iCBUi}
		Let ${\zeta_\imath} \in X_{\imath}$ and $(b_1, b_2) \in B_{\Iblack} \times B$. 
	\begin{enumerate}
		\item	
		There is a unique element $u =b_1 \diamondsuit^\imath_{\zeta_\imath} b_2  \in \Uidot$ such that 
			\[
				u (\eta^{\bullet}_\lambda \otimes \eta_\mu) 
				= (b_1 \diamondsuit_{\zeta_\imath} b_2)_{\wb\lambda,\mu}^\imath				\in L^{\imath} (\lambda, \mu), 
			\]
for all $\lambda, \mu \gg0$ with $\overline{\wb \lambda+\mu} = {\zeta_\imath}$.
		\item	The element $b_1 \diamondsuit^\imath_{\zeta_\imath} b_2$ is $\ipsi$-invariant.
		\item	The set $\Bdot^\imath =\{ b_1 \diamondsuit^\imath_{\zeta_\imath} b_2 \big \vert {\zeta_\imath} \in X_\imath, (b_1, b_2) \in B_{\Iblack} \times B \}$ 
		forms a $\Qq$-basis of $\Uidot$ and an $\mA$-basis of ${}_\mA\Uidot$.
	\end{enumerate}
\end{thm}

\begin{proof}  
	The proof is the same as for \cite[Theorem 6.17]{BW18b} once Proposition~\ref{prop:props} is available.
\end{proof}

\begin{rem} 
Let us illustrate the dependence on parameters for small $\nu$ by a simple example. 
We consider the quantum symmetric pair of type AIV of rank one.
We have $\imath(B) = F + \vs EK^{-1}$ (with $\kappa =0$). Let us write $\vs = q^{-2} \overline{\vs} = q^{-1}\sum_{i \in \Z} a_i q^{i}$ with $a_i =a_{-i} \in \Z$. We have the contraction map
\begin{align*}
\pi: L(2) &\longrightarrow L(0) \\
\eta_2 \mapsto \eta_0,\quad F \eta_2 & \mapsto 0, \quad
F^{(2)} \eta_2 \mapsto  - \vs \eta_0.
\end{align*}

The $\imath$canonical basis of $L(2)$ is of the following form 
\[
\eta_2^\imath = \eta_2, \quad (F \eta_2)^\imath= F \eta_2, \quad  (F^{(2)} \eta_2)^\imath = F^{(2)} \eta_2 + (\sum_{i < 1} a_i q^{i-1} - \sum_{i > 1} a_i q^{1-i}) \eta_2.
\]
Therefore we see that 
\[
\pi((F^{(2)} \eta_2)^\imath) = - \sum_{i>1} a_i (q^{i} + q^{-i}).
\]
Hence the map $\pi: L(2) \longrightarrow L(0)$ is generally not a morphism of based $\Ui$-modules.
\end{rem}

Recall the linear map $p_{\imath,\lambda}:   \Uidot \one_{\overline{\lambda}} \longrightarrow  \Padot \one_{\lambda}$ from \eqref{eq:comp3}. 

\begin{cor}   \cite[Corollaries 6.19, 6.20]{BW18b}
   \label{cor:generator}
The $\mA$-algebra ${}_\mA \Uidot$ is generated by $1  \diamondsuit_{{\zeta_\imath}}^{\imath} F^{(a)}_{i}$ $(i \in \I)$ 
and $E^{(a)}_{j} \one_{{\zeta_\imath}}$ $(j \in \Iblack)$ for ${\zeta_\imath} \in X_\imath$ and $a \ge 0$.
Moreover, ${}_\mA \Uidot$ is a free $\mA$-module such that $\Uidot =\Qq \otimes_{\mA} {}_\mA \Uidot$. 

For $\lambda \in X$, the map $p_{\imath,\lambda}: {}_\mA \Uidot \one_{\overline{\lambda}} \longrightarrow {}_\mA\Padot \one_{\lambda}$ 
is an isomorphism of (free) $\mc A$-modules. 
\end{cor}

Recall from Definition~\ref{def:AUi} the $\mA$-subalgebra $\aAp \Uidot \subset {}_\mA \Uidot$.

\begin{cor}\label{cor:AUi}
We have $\aAp \Uidot = {}_\mA \Uidot$; that is, the $\mA$-algebra ${}_\mA \Uidot$ is generated by the $\imath$divided powers $B^{(a)}_{i,\zeta_\imath}$ $(i \in \I)$ 
and $E^{(a)}_{j} \one_{{\zeta_\imath}}$ $(j \in \Iblack)$ for ${\zeta_\imath} \in X_\imath$ and $a \ge 0$.
\end{cor}

\begin{proof}
The equality follows from Corollary~\ref{cor:generator} and the surjectivity of the restriction $ p_{\imath,\lambda}: \aAp \Uidot \one_{\overline{\lambda}} \longrightarrow {}_\mA\Padot \one_{\lambda}$. The rephrase follows by Definition~\ref{def:AUi} of $\aAp \Uidot$. 
\end{proof}

\subsection{The bilinear form}
Following \cite[Definition 6.15]{BW18b}, we define a symmetric bilinear form $(\cdot , \cdot): \Uidot \times \Uidot \rightarrow \Q(q)$ via a limit of the corresponding bilinear forms (defined using the anti-involution $\wp$) on the projective system of $\Ui$-modules $L^\imath (\la, \mu)$. 

The next theorem generalizes \cite[Theorem 6.27]{BW18b} to $\Uidot$ associated with more general parameters $\vs_i$ and to the Kac-Moody setting.
\begin{thm}  
  \label{thm:orth}
	The $\imath$canonical basis $\Bdot^\imath$ of $\Uidot$ is almost orthonormal in the following sense: 
	for ${\zeta_\imath}, \zeta_\imath' \in X_\imath$ and $(b_1, b_2), (b'_1, b'_2) \in \B_{\Iblack} \times \B$, we have 
\[
(b_1 \diamondsuit^\imath_{\zeta_\imath} b_2, b'_1 \diamondsuit^\imath_{\zeta_\imath'} b'_2) \equiv \delta_{{\zeta_\imath}, \zeta_\imath'} \delta_{b_1, b'_1} \delta_{b_2, b'_2}, \quad \text{ mod } q^{-1} {\bf A}.
\]
In particular, the bilinear form $(\cdot , \cdot)$ on $\Uidot$ is non-degenerate. 
\end{thm}

\begin{rem}
A noticeable feature in the definition of the bilinear forms on $\Udot$, or on $\Uidot$ as in \cite[\S6.6]{BW18b} (where it was assumed that $\vs_i \in \pm q^\Z$), is the adjunction $\wp$. 

For $\vs_i$ in \eqref{parameters} satisfying the general relation \eqref{vs2} in general, the $\U$-automorphism $\wp$ does {\em not} restrict to an automorphism of $\Ui$, or $\Uidot$. However, since $L^\imath (\la, \mu)$ is a $\U$-module, the $\wp$-twisted action of $\Ui$ is always well defined. Hence the definition of the bilinear form $(\cdot , \cdot)$ on $\Uidot$ based on the projective system of $\Ui$-modules $L^\imath (\la, \mu)$ remains valid, and the proof of  \cite[Theorem 6.27]{BW18b} can be carried here.

While the adjunction induced by $\wp$ no longer holds in general, if the anti-involution $\wp$ preserves the algebra $\Ui$ as in Proposition~\ref{prop:rho} (3), we shall still have the induced adjunction.
\end{rem}


\end{document}